\documentclass[10pt,a4paper]{article}
\usepackage{amsfonts}
\usepackage{amsmath}
  \title{\textbf{Curvature of a class of indefinite globally framed $f$-manifolds}}
 \author{\begin{tabular}{cc}
  Letizia Brunetti and Anna Maria Pastore\\
     \end{tabular}}
   \date{}
\usepackage{graphicx}
\usepackage{latexsym,amssymb}
\usepackage{amsmath}
\newtheorem{theorem}{Theorem}[section]

\newtheorem{remark}[theorem]{Remark}

\newtheorem{corollary}[theorem]{Corollary}
\newtheorem{definition}[theorem]{Definition}
\newtheorem{example}[theorem]{Example}
\newtheorem{lemma}[theorem]{Lemma}

\newtheorem{proposition}[theorem]{Proposition}
\newenvironment{proof}[1][Proof]{\emph{#1.} }{\hfill{\mbox{$\square$}}\medskip}
\begin{document}

\maketitle

\begin{abstract}
We present a compared analysis  of some properties of indefinite almost $\mathcal{S}$-manifolds 
and indefinite $\mathcal{S}$-manifolds. We give some characterizations in terms of the Levi-Civita connection and of the characteristic vector fields.  We study the sectional and ${\varphi}$-sectional curvature of indefinite almost $\mathcal{S}$-manifolds and state an expression of the curvature tensor field for the indefinite $\mathcal{S}$-space forms. We analyse the sectional curvature of indefinite $\mathcal{S}$-manifold in which the number of the  spacelike characteristic vector fields is equal to that of the timelike characteristic vector fields. 
Some examples are also described.
\end{abstract}


\section{Introduction} In the framework of Riemannian geometry, almost $S$-manifolds and $S$-manifolds represent a natural generalization of contact and Sasaki manifolds, respectively. Such manifolds have been extensively studied by several authors and from different points of view (\cite{Bl1, Bl2, Bl3, Di, Di1, DIP}). On the other hand, also Sasakian manifolds with semi-Riemannian metric have been considered (\cite{[DB], Calin1, T}), and in recent works many authors, (for example, in \cite{DS}, K.L.\ Duggal and B.\ Sahin) study lightlike submanifolds of indefinite Sasakian manifolds. Indefinite $\mathcal{S}$-manifolds are natural generalizations of indefinite Sasaki manifolds. Moreover many spacetime manifolds can be endowed with $f$-structures (\cite{Dug0}).

After a first section on $f$-structures and indefinite metric $g.f.f$-structures, in section 3, we carry out an in-depth study of the indefinite (almost) $\mathcal{S}$-manifolds.
In section 4 we describe two examples of $6$-dimensional indefinite $\mathcal{S}$-manifolds having two characteristic vector fields which are both spacelike or both timelike. A third example is a Lorentzian indefinite $\mathcal{S}$-manifold of dimension $4$ with two characteristic vector fields of different causal type.
In section 5, after some Lemmas, we prove that the ${\varphi}$-sectional curvatures completely determine the sectional curvatures. Then, we find an expression of the curvature tensor field ${R}$ which characterizes the indefinite $\mathcal{S}$-space forms, that is indefinite $\mathcal{S}$-manifolds with constant ${\varphi}$-sectional curvature.  
Then, in section 6, we consider the curvature of special indefinite $\mathcal{S}$-manifold in which the number of the characteristic vector fields is even with an equal number of spacelike and timelike characteristic vector fields; we prove that the special indefinite $\mathcal{S}$-manifold described in the third example in section 4 turns out to be an indefinite $\mathcal{S}$-space form whose ${\varphi}$-sectional curvature vanishes.

 All manifolds and tensor fields are assumed to be smooth.

\emph{Acknowledgments}. The authors are grateful to Prof. S. Ianus for discussions about the topic of this paper during his stay at the University of Bari and the stay of the first author at the University of Bucharest.

\section{Indefinite metric $f$-structure}
We recall that an $f$-structure on a manifold $M$ is a non null $(1,1)$-tensor field $\varphi$ on $M$ of constant rank such that $\varphi\,^{3}+\varphi=0$.
A manifold $M$, provided with an $f$-structure, is said to be an
$f$\emph{-manifold}, and it is known that $TM$ splits into two complementary subbundles $\operatorname{Im}\varphi$ and $\ker\varphi$ and that the restriction of $\varphi$ to $\operatorname{Im}\varphi$ determines a complex structure on it and the rank of $\varphi$ is even. 
An interesting case of $f$-structure occurs when $\ker\varphi$ is parallelizable for which there exist global vector fields $\xi_{\alpha}$, ${\alpha\in\{1,\ldots,r\}}$, with their dual $1$-forms $\eta^{\alpha}$, satisfying:
$\varphi^{2}=-I+\sum_{\alpha =1}^r \eta^{\alpha}\otimes\xi_{\alpha}$, and $\eta^{\alpha}(\xi_{\beta})=\delta_{\beta}^{\alpha}$.
Such an $f$-structure is called an $f$-\emph{structure with parallelizable kernel} or \emph{globally
framed }$f$-\emph{structure}, briefly denoted $g.f.f$-\emph{structure} (\cite{GY}). Moreover, a  manifold $M$ endowed with a $g.f.f$-structure is
called a $g.f.f$-\emph{manifold}, and it is denoted with $(
M,\varphi,\xi_{\alpha},\eta^{\alpha})$; the vector fields $\xi_{\alpha}$, $(\alpha=1,....,r)$, are
called \emph{characteristic vector fields}.

It is also known that an $f$-structure, on a manifold $M$, is called \emph{normal} if the tensor field $N=N_{\varphi}+2\sum_{\alpha =1}^r d \eta^{\alpha}\otimes\xi_{\alpha}$ vanishes, where $N_{\varphi}$ is the Nijenhuis torsion of $\varphi$.

\begin{definition}\emph{
Let $(M,\varphi)$ be a
$(2n+r)$-dimensional $f$-manifold and $g$ a semi-Riemannian metric on
${M}$ with index $\nu$, $0<\nu<2n+r$. Then, the pair $(
\varphi,g)$ is said to be an {\it indefinite metric $f$-structure}, and the triple $(M%
,\varphi,g)$ is called an {\it indefinite metric $f$-manifold}, if $\varphi$ is skew-symmetric with
respect to $g$, that is, for any $X,Y\in\Gamma(TM)
$:}
\begin{equation*}
g(\varphi X,Y)+g(X,\varphi
Y)=0.
\end{equation*}
\end{definition}
\begin{definition}\emph{
Let $(M^{2n+r},\varphi,\xi_{\alpha
},\eta^{\alpha})$ be a $g.f.f$
-manifold, and $g$ a semi-Riemannian metric on $M$ with index
$\nu$, $0<\nu<2n+r$. Then, we say that the two structures are
{\it compatible} if for any $X,Y\in \Gamma (TM)$ 
\begin{equation}\label{eqz01}
g(\varphi X,\varphi Y)=g(
X,Y)-\sum_{\alpha=1}^{r}\varepsilon_{\alpha}\eta^{\alpha}(
X)  \eta^{\alpha}(Y),\quad  \varepsilon_{\alpha}g(X,\xi_{\alpha})=\eta
^{\alpha}(X)\quad {\rm for\; any }\; \alpha\in\{1,\ldots,r\},
\end{equation}
where $\varepsilon_{\alpha}=\pm1$ according to whether $\xi_{\alpha}$ is
spacelike or timelike.
Then $(M^{2n+r},\varphi,\xi_{\alpha},\eta
^{\alpha},g)$ is called an {\it indefinite metric
$g.f.f$-manifold}.}
\end{definition}

We shall use the Einstein convention  omitting the sum symbol for repeated indices above and below, writing, e.g., $\varepsilon_{\alpha}\eta^{\alpha}(
X)  \eta^{\alpha}(Y)$ to mean $\sum_{\alpha=1}^{r}\varepsilon_{\alpha}\eta^{\alpha}(
X)  \eta^{\alpha}(Y)$.

Observe that if $g$ is a semi-Riemannian metric on a
$g.f.f$-manifold $(M,\varphi,\xi_{\alpha
},\eta^{\alpha})$ compatible with the $f$-structure
$\varphi$, then the pair $(\varphi,g)$ is necessarily an
indefinite metric $f$-structure. The fundamental $2$-form $\Phi$ is defined putting $\Phi(X,Y)=g(X,\varphi Y)$, for any $X,Y\in\Gamma(TM)$. Let $(M,\varphi,\xi_{\alpha
},\eta^{\alpha})$, with $\alpha=1,\ldots,r$, be a $g.f.f$%
-manifold, and $g$ a compatible semi-Riemannian metric on $M$. We
know that the orthogonal decomposition $TM=\operatorname{Im}\varphi\oplus \ker\varphi$ holds, and that the induced structure
$J$ on $\operatorname{Im}\varphi$ is an almost complex structure; then
$(\operatorname{Im}\varphi,g=g|
_{\operatorname{Im}\varphi},J)$ is a indefinite Hermitian
distribution and the only possible signatures of $g$ are $(
2p,2q)$ with $p+q=n$; therefore $g$ cannot be a Lorentz metric, for $n>1$.
  We shall denote
$\operatorname{Im}\varphi$ and $\ker\varphi$ with
$\mathfrak{D}$ and $\mathfrak{D}^{\bot}$ respectively and for a section of $\mathfrak{D}$ ( $\mathfrak{D}^{\bot}$) we will write $X\in\mathfrak{D}$ or $X\in\Gamma(\mathfrak{D})\,$  ( $X\in\mathfrak{D}^{\bot}$ or $X\in\Gamma(\mathfrak{D}^{\bot})$). 

We recall the following result due to A.\ Bejancu and K.L.\ Duggal (\cite{[DB]}).
\begin{theorem}
Let $(M,\varphi,\xi_{\alpha
},\eta^{\alpha})$, $\alpha=1,\ldots,r$, be a $g.f.f.$%
-manifold and $h_{0}$ a semi-Riemannian metric on $M$; 
we suppose that
$\{\xi_{\alpha}\}_{1\leq\alpha\leq r}$ are $h_{0}$-orthonormal
and that $h_{0}(\xi_{\alpha},\xi_{\alpha})
=-\varepsilon_{\alpha}$, for any $\alpha \in \{1,\ldots
,r\}$. Then there exists a symmetric tensor field $g$
of type $(0,2)$ on $M$ satisfying {\rm (\ref{eqz01})}.
\end{theorem}

Now, with a standard computation as in the Riemannian setting (\cite{Bl1}), one can prove the following results.
\begin{proposition}
Let $(
M,\varphi,\xi_{\alpha},\eta^{\alpha},g)$ be
an indefinite metric $g.f.f$-manifold. Then, the Levi-Civita connection
satisfies the following equality, for any $X,Y,Z\in \Gamma(TM)$:
\begin{align}
2g((\nabla_{X}\varphi)Y,Z)   &
=3d\Phi(X,\varphi Y,\varphi Z)  -3d\Phi(
X,Y,Z) \label{eq-01} +g(N(Y,Z),\varphi X)+\varepsilon_{\alpha}N_{\alpha}^{(  2)
}(Y,Z)  \eta^{\alpha}(X) \nonumber \\
&\quad +2\varepsilon_{\alpha}d\eta
^{\alpha}(\varphi Y,X)  \eta^{\alpha}(Z)
 -2\varepsilon_{\alpha}d\eta
^{\alpha}(\varphi Z,X)  \eta^{\alpha}(
Y),
\end{align}
where  
$N_{\alpha}^{(2)  }(X,Y)=(\mathcal{L}
_{\varphi X}\eta^{\alpha}) (Y)  -(
\mathcal{L}_{\varphi Y}\eta^{\alpha}) (X)=2d\eta^{\alpha}(\varphi X,Y)  -2d\eta
^{\alpha}(\varphi Y,X)$.
\end{proposition}
\begin{proposition}\label{Nphi3}
Let $(M,\varphi,\xi_{\alpha},\eta^{\alpha},g)$ be an indefinite metric $g.f.f$-manifold. Then the following statements hold:
\begin{enumerate}
\item[a)]$(\mathcal{L}_{\xi_{\alpha}}\Phi)(X,Y)=(\mathcal{L}_{\xi_{\alpha}}g)(X,\varphi Y)+g(X,(\mathcal{L}_{\xi_{\alpha}}\varphi)Y)$, for any $\alpha\in\{  1,\ldots,r\}  $.
\item[b)] $(\nabla_{X}\Phi)(Y,Z)=g(Y,(\nabla_{X}\varphi)Z)$, for any $X,Y,Z\in\Gamma(TM)$.
\item[c)] If $\mathcal{L}_{\xi_{\alpha}}\varphi=0$, then $\eta^{\beta}[\varphi Z,\xi_{\alpha}]=0$, for any $\beta\in\{  1,\ldots,r\}$.
\item[d)] $N=0\Rightarrow N^{(2)}_{\alpha}=0$, for any $\alpha\in\{  1,\ldots,r\}  $.
\end{enumerate}
\end{proposition}
Between the indefinite metric $g.f.f$-manifolds, we can define the following classes.
\begin{definition}\emph{
Let $(M^{2n+r},\varphi,\xi_{\alpha},\eta^{\alpha},g)$ be an indefinite metric $g.f.f$-manifold. $M$ is
called {\it indefinite $\mathcal{K}$-manifold} if it is normal and $d\Phi=0$.}
\end{definition}
In this case $\mathcal{L}_{{\xi}_{\alpha}}\Phi=i_{{\xi}_{\alpha}}d\Phi+di_{{\xi}_{\alpha}}\Phi=0$, therefore, from a) of Proposition \ref{Nphi3}, we obtain that $\mathcal{L}_{\xi_{\alpha}}\varphi=0$ if and only if the characteristic vector fields $\xi_{\alpha}$ are Killing.
Two subclasses of indefinite $\mathcal{K}$-manifolds are those of indefinite $\mathcal{C}$-manifolds and indefinite $\mathcal{S}$-manifolds, that are defined as follows: an indefinite $\mathcal{K}$-manifold is
called {\it indefinite $\mathcal{C}$-manifold} if $d\eta^{\alpha}=0$ for any
$\alpha\in\{1,\ldots,r\}$, while it is
called {\it indefinite $\mathcal{S}$-manifold} if $d\eta^{\alpha}=\Phi$ for any
$\alpha\in\{1,\ldots,r\}$. 

\section{Indefinite $\mathcal{S}$-manifolds}

The properties of (almost) $\mathcal{S}$-manifolds (with Riemannian metric) are studied in \cite{DIP} and in \cite{Bl1}. Now, we discuss indefinite (almost) $\mathcal{S}$-manifolds and their properties.

\subsection{Indefinite almost $\mathcal{S}$-manifolds}

\begin{definition}\emph{
Let $(
M^{2n+r},\varphi,\xi_{\alpha},\eta^{\alpha},
g)$ be an indefinite metric $g.f.f$-manifold. $M$ is
called {\it indefinite almost $\mathcal{S}$-manifold} if $d\eta^{\alpha}=\Phi$ for any}
$\alpha\in\{1,\ldots,r\}.$
\end{definition}
\begin{lemma}
\label{LemNalfa} Let $(
M,\varphi,\xi_{\alpha},\eta^{\alpha},g)$
be an indefinite almost $\mathcal{S}$-manifold. Then the tensor fields
$N_{\alpha}^{(2)}$ vanish and for any $X,Y\in\Gamma(\mathfrak{D})$ and
$\alpha\in\{1,\ldots,r\}$, we have
\[ \eta^{\alpha}[\varphi X,Y]  =\eta^{\alpha
}[\varphi Y,X]\]
\end{lemma}%
\begin{proof}
For $\alpha\in\{
1,\ldots,r\}$, we have $N_{\alpha}^{(2)}(X,Y)=2d\eta^{\alpha
}(\varphi X,Y)  -2d\eta^{\alpha}(
\varphi Y,X) =2\Phi(\varphi X,Y)  -2\Phi(\varphi Y,X)=0$. Then, for any $X,Y\in\Gamma(\mathfrak{D})$, $2d\eta^{\alpha
}(\varphi X,Y)=-\eta^{\alpha}([\varphi X,Y])$ implies $\eta^{\alpha}[\varphi X,Y]  =\eta^{\alpha
}[\varphi Y,X]$.
\end{proof}
\begin{proposition}
\label{prop condizioniS}
Let $(M,\varphi,\xi_{\alpha},\eta^{\alpha}%
,g)$ be an indefinite almost $\mathcal{S}$-manifold and $\bar{\eta}:=\sum_{\alpha=1}^{r}\varepsilon_{\alpha}\eta^{\alpha}
$. Then, the
following statements hold:
\begin{align}
2g((\nabla_{X}\varphi)Y,Z)   &
=g(N(Y,Z),\varphi X)  +2g(
\varphi Y,\varphi X)  \bar{\eta}(Z) -2g(\varphi Z,\varphi X)
\bar{\eta}(Y),\label{eq00} 
\end{align}
\begin{equation}
\nabla_{\xi_{\alpha}}\varphi=0\quad\quad\textrm{for all}\;\;\alpha
\in\{1,\ldots,r\}, \qquad \nabla_{\xi_{\alpha}}\xi_{\beta}=0\quad\quad\textrm{for all}\;\;
\alpha,\beta\in\{1,\ldots,r\}.\label{eq01}%
\end{equation}
\end{proposition}%
 \begin{proof} Equation (\ref{eq00}) follows from (\ref{eq-01})
using $d\Phi=0$, $N_{\alpha}^{(2)
}=0$ and $d\eta
^{\alpha}=\Phi$, for $\alpha\in\{1,\ldots,r\}$.
 Then, putting $X=\xi_{\alpha}$, we obtain $\nabla_{\xi
_{\alpha}}\varphi=0$.

Hence, we have $0=(\nabla_{\xi_{\alpha}}\varphi)  (
\xi_{\beta})  =-\varphi(\nabla_{\xi_{\alpha}}\xi_{\beta})$  ,
therefore $\nabla_{\xi_{\alpha}}\xi_{\beta}\in\mathfrak{D}^{\bot}$, which implies that $[\xi_{\alpha},\xi_{\beta}]
\in\mathfrak{D}^{\bot}$.
On the other hand, for any $\gamma\in\{  1,\ldots,r\}  $%
\begin{align*}
0  & =\Phi(\xi_{\alpha},\xi_{\beta})  =d\eta
^{\gamma}(\xi_{\alpha},\xi_{\beta}) =-\frac{1}{2}\eta^{\gamma}[  \xi_{\alpha},\xi
_{\beta}]  =-\frac{1}{2}\varepsilon_{\gamma}g( [
\xi_{\alpha},\xi_{\beta}]  ,\xi_{\gamma})  .
\end{align*}
Therefore $[\xi_{\alpha},\xi_{\beta}]\in\mathfrak{D}\cap\mathfrak{D}^{\bot}$ and we obtain $[{\xi
}_{\alpha},{\xi}_{\beta}]  =0$ and ${\nabla}_{{\xi}_{\alpha
}}{\xi}_{\beta}={\nabla}_{{\xi}_{\beta}}{\xi}_{\alpha}$.
Now we check that ${\nabla}_{{\xi}_{\alpha}}{\xi}_{\beta}%
\in\mathfrak{D}$, that is, for any $\gamma\in\{
1,\ldots,r\}$, $g(  {\nabla}_{{\xi}_{\alpha}}{\xi}_{\beta},{\xi
}_{\gamma})  =0$.
Being ${g}({\xi}_{\beta},{\xi}_{\gamma})
=\varepsilon_{\beta}\delta_{\beta\gamma}$ and using the covariant derivative with respect to ${\xi}_{\alpha}$, we find
${g}(  {\nabla}_{{\xi}_{\alpha}}{\xi}_{\beta},{\xi
}_{\gamma})  +{g}(  {\xi}_{\beta},{\nabla}_{{\xi
}_{\alpha}}{\xi}_{\gamma})  =0$,
 and, covariantly differentiating ${g}({\xi}_{\alpha},{\xi}_{\gamma})
=\varepsilon_{\alpha}\delta_{\alpha\gamma}$ with respect to ${\xi}_{\beta}$,
we obtain ${g}(  {\nabla}_{{\xi}_{\beta}}{\xi}_{\alpha},{\xi
}_{\gamma})  +{g}(  {\xi}_{\alpha},{\nabla}_{{\xi
}_{\beta}}{\xi}_{\gamma})  =0$.
From the last two equations, using ${\nabla}_{{\xi
}_{\alpha}}{\xi}_{\beta}={\nabla}_{{\xi}_{\beta}}{\xi}%
_{\alpha}$, we have ${g}(  {\xi}_{\beta},{\nabla}_{{\xi}_{\alpha}}{\xi
}_{\gamma})  ={g}(  {\xi}_{\alpha},{\nabla}_{{\xi
}_{\beta}}{\xi}_{\gamma})$.
Therefore, ${g}(  {\nabla}_{{\xi}_{\alpha}}{\xi}_{\beta},{\xi
}_{\gamma})={g}(  {\xi}_{\alpha},{\nabla}%
_{{\xi}_{\gamma}}{\xi}_{\beta})  ={g}(  {\xi
}_{\alpha},{\nabla}_{{\xi}_{\beta}}{\xi}_{\gamma})=-{g}(  {\nabla}_{{\xi}_{\beta}}{\xi}_{\alpha},
{\xi}_{\gamma})  =-{g}(  {\nabla}_{{\xi}_{\alpha}}%
{\xi}_{\beta},{\xi}_{\gamma})$,
from which ${g}(  {\nabla}_{{\xi}_{\alpha}}{\xi}_{\beta
},{\xi}_{\gamma})  =0$ follows. This result and $\nabla_{\xi_{\alpha}}\xi_{\beta}\in\mathfrak{D}^{\bot}$ imply  ${\nabla}_{{\xi}_{\alpha}}{\xi
}_{\beta}=0.$%
\end{proof}
\begin{proposition}
Let $({M},{\varphi
},{\xi}_{\alpha},{\eta}^{\alpha},{g})  $ be an indefinite
almost $\mathcal{S}$-manifold. Then
\begin{enumerate}
\item[a)] for any $\alpha\in\{  1,\ldots,r\}$ the operator
$h_{\alpha}=\frac{1}{2}\mathcal{L}_{{\xi}_{\alpha}}{\varphi}$ is self-adjoint,
\item[b)] for any $\alpha,\beta\in\{  1,\ldots,r\}  $, $h_{\alpha
}({\xi}_{\beta})  =0$,
\item[c)] for any $\alpha\in\{  1,\ldots,r\}$, $h_{\alpha
} \circ \varphi+ \varphi \circ h_{\alpha
}=0$.  
\end{enumerate}
\end{proposition}
\begin{proof}
As first step, using (\ref{eq01}),
 for any $X,Y\in\Gamma(T
{M})$ and any $\alpha\in\{1,\ldots,r\}$, we easily obtain,
\begin{equation*}
{g}(  (  \mathcal{L}_{{\xi}_{\alpha}}{\varphi})
X,Y)  =\varepsilon_{\alpha}(  -(  {\varphi}X)
(  {\eta}^{\alpha}(  Y)  )  +{\eta}^{\alpha
}(  {\nabla}_{{\varphi}X}Y+{\nabla}_{X}(  {\varphi
}Y))).
\end{equation*} 
It follows that
\begin{align*}
2{g}( h_{\alpha}(X),Y)  -2{g}(  h_{\alpha}(Y),X)   
&  =-\varepsilon_{\alpha}(  {\varphi}X)  (  {\eta
}^{\alpha}(  Y)  )  +\varepsilon_{\alpha}{\eta}^{\alpha
}[  {\varphi}X,Y]  +\varepsilon_{\alpha}(  {\varphi}Y)  (
{\eta}^{\alpha}(  X)  )\\
&  \quad  -\varepsilon_{\alpha}{\eta
}^{\alpha}[  {\varphi}Y,X]  =-\varepsilon_{\alpha}(  \mathcal{L}_{{\varphi}X}{\eta}^{\alpha})  (
Y)  +\varepsilon_{\alpha}(  \mathcal{L}_{{\varphi}Y}{\eta}^{\alpha})
(  X) =0.
\end{align*}
Obviously, for any $\alpha,\beta\in\{  1,\ldots,r\}  $ we have $h_{\alpha}(  {\xi}_{\beta})  =0$ and
finally
\begin{align*}
2(h_{\alpha} \circ \varphi+ \varphi \circ h_{\alpha})(X)&=\mathcal{L}_{\xi_{\alpha}}(\varphi^{2}X)-\varphi(\mathcal{L}_{\xi_{\alpha}}(\varphi X))+\varphi(\mathcal{L}_{\xi_{\alpha}}(\varphi X)-\varphi(\mathcal{L}_{\xi_{\alpha}} X))\\
&=\xi_{\alpha}(\eta^{\beta}(X))\xi_{\beta}-\eta^{\beta}[\xi_{\alpha},X]\xi_{\beta}=0
\end{align*}
for any $\alpha\in\{1,\ldots,r\}$ and any $X\in\Gamma(TM)$.
\end{proof}
\begin{proposition}\label{Propabb}
Let $(  {M},{\varphi
},{\xi}_{\alpha},{\eta}^{\alpha},{g})$ be an indefinite
almost $\mathcal{S}$-manifold. Then, for any $X,Y\in\Gamma( T
{M})$, the following properties hold:
\begin{enumerate}
\item[a)] ${\varphi}(  N(  X,Y)  )  +N(
{\varphi}X,Y)  =2{\eta}^{\alpha}(  X)  h_{\alpha
}(  Y)  $,
\item[b)] $N(X,Y)  \in\mathfrak{D}$.
 \end{enumerate}
\end{proposition}
 \begin{proof}
Using Lemma \ref{LemNalfa}, we obtain
\begin{align*}
{\varphi}(  N(  X,Y)  )  +N(  {\varphi
}X,Y)  
&  =-(  \mathcal{L}_{{\varphi}Y}{\eta}^{\alpha})  (
X)  {\xi}_{\alpha}+(  \mathcal{L}_{{\varphi}X}{\eta
}^{\alpha})  (  Y)  {\xi}_{\alpha}+{\eta}^{\alpha}(  X)  (
\mathcal{L}_{{\xi}_{\alpha}}{\varphi})  (  Y) \\
& =2{\eta}^{\alpha}(  X)  h_{\alpha}(  Y).
\end{align*}
Now, we observe that for any $\alpha\in\{1,\ldots,r\}$ we have
 $[ {\xi}_{\alpha},\mathfrak{D}]  \subset\mathfrak{D}$,
in fact, if $\beta\in\{1,\ldots,r\}$ and $X\in\Gamma(T{M})$, we have
${\eta}^{\beta}[  {\xi}_{\alpha},{\varphi}X]=-2d{\eta}^{\beta}(  {\xi}_{\alpha},{\varphi}X)=0$
and in particular, if $X\in\mathfrak{D}$ and $\alpha=\beta$, we get ${\eta}^{\alpha}[  {\xi}_{\alpha},X]  =0$.
So, if $Z\in\mathfrak{D}$ then $N({\xi}_{\alpha},Z)  =-[  {\xi}_{\alpha},Z]  -{\varphi}[  {\xi
}_{\alpha},{\varphi}Z]  \in\mathfrak{D}$.
It is easy to check that $N(  {\xi}_{\alpha},{\xi}_{\beta
})  =0$ for any $\alpha,\beta\in\{  1,\ldots,r\}$; therefore, we have that $N(  {\xi}_{\alpha},X)  \in
\mathfrak{D}$ for any $X\in\Gamma(T{M})$.
Finally, applying a), we have ${g}(  N(  {\varphi}X,Y)  ,{\xi}_{\alpha})
=2{\eta}^{\beta}(  X)  {g}(  h_{\beta}(
Y)  ,{\xi}_{\alpha})  =0.$
Hence, if $X,Y\in\Gamma(T{M})$, we get $N(  X,Y)  =-N(  {\varphi}^{2}X,Y)  +{\eta
}^{\alpha}(  X)  N(  {\xi}_{\alpha},Y)$, and being $N(  {\varphi}^{2}X,Y)  \in\mathfrak{D}$ and $N(  {\xi}_{\alpha
},Y)  \in\mathfrak{D}$, we conclude that $N(
X,Y)  \in\mathfrak{D}$.
 \end{proof}
\begin{proposition}\label{nablaxi}
\label{nablax} Let $(  {M}%
,{\varphi},{\xi}_{\alpha},{\eta}^{\alpha},{g})  $ be an
indefinite almost $\mathcal{S}$-manifold. For any $X\in\Gamma(T{M})$ and for any $\alpha\in\{  1,\ldots,r\}
$,
\[
{\nabla}_{X}{\xi}_{\alpha}=-\varepsilon_{\alpha}{\varphi}(  X)
-{\varphi}(  h_{\alpha}X)  .
\]
\end{proposition}
\begin{proof}
Putting $X={\xi}_{\alpha}$ in a) of Proposition \ref{Propabb}, we have that for any $Z,Y\in \Gamma(T
{M})$%
\begin{align*}
{g}(N({\xi}_{\alpha},Y),{\varphi}Z)
& =-{g}({\varphi}(N({\xi}_{\alpha},Y)
),Z) =-2{\eta}^{\beta}({\xi}_{\alpha}){g}(
h_{\beta}(Y),Z) =-2{g}(h_{\alpha}(Y),Z).
\end{align*}
Moreover, applying (\ref{eq00}) of Proposition \ref{prop condizioniS}, for any $\alpha\in\{1,\ldots,r\}$ we find:
\begin{align*}
{g}(-{\varphi}({\nabla}_{X}{\xi}_{\alpha
}),Z)   & =\frac{1}{2}{g}(N({\xi}_{\alpha},Z)
,{\varphi}X)  -{g}({\varphi}Z,{\varphi}X)
{\eta}({\xi}_{\alpha}) \\
& =-{g}(h_{\alpha}(Z),X)  -\varepsilon
_{\alpha}{g}(Z,X)  +\varepsilon_{\alpha}\varepsilon_{\beta}{\eta}^{\beta}(  X)  {\eta}^{\beta
}(Z) \\
& ={g}(-h_{\alpha}(X)  -\varepsilon_{\alpha
}X+\varepsilon_{\alpha}{\eta}^{\beta}(X)  {\xi}_{\beta
},Z)  ,
\end{align*}
then ${\varphi}(  {\nabla}_{X}{\xi}_{\alpha})  =h_{\alpha
}(  X)  +\varepsilon_{\alpha}X-\varepsilon_{\alpha}{\eta
}^{\beta}(  X)  {\xi}_{\beta}$,
and, applying ${\varphi}$, we complete the proof.
Note that ${\nabla}_{X}{\xi}_{\alpha}\in\mathfrak{D}$.
\end{proof} 
\begin{proposition}
Let $({M},{\varphi},{\xi}_{\alpha},{\eta}^{\alpha},{g})$ be
an indefinite almost $\mathcal{S}$-manifold. For $X,Y\in\Gamma(T{M})$, we have
\[
(  {\nabla}_{X}{\varphi})  (  Y)  +(
{\nabla}_{{\varphi}X}{\varphi})  (  {\varphi
}Y)  =2{g}(  {\varphi}X,{\varphi}Y)  \bar{\xi
}+\bar{\eta}(  Y)  {\varphi}^{2}(  X)  -{\eta
}^{\alpha}(  Y)  h_{\alpha}(  X).
\]
where $\bar{\xi}:=\sum
_{\alpha=1}^{r}{\xi}_{\alpha}$ and 
 $\bar{\eta}(X)=g(X,\bar{\xi})$, for any $X\in\Gamma(TM)$.
\end{proposition}
\begin{proof}
Using (\ref{eq00}), Proposition \ref{Propabb} and Proposition \ref{nablaxi}, for any
$X,Y,Z\in\Gamma(T{M})$ we have
\begin{align*}
2{g}(  (  {\nabla}_{X}{\varphi}) (
Y)  ,Z)+2{g}( (  {\nabla}_{{\varphi}%
X}{\varphi}) (  {\varphi}Y)  ,Z)   
&  =-{g}(  {\varphi}(  N(  Y,Z)  )
+N(  {\varphi}Y,Z)  ,X)  \\
& \quad +4{g}(  {\varphi}Y,{\varphi}X)
\bar{\eta}(  Z)  -2{g}(  {\varphi}Z,{\varphi
}X)  \bar{\eta}(  Y) \\
&  =-2{g}(  Z,{\eta}^{\alpha}(  Y)  h_{\alpha}(
X) ) +4{g}(  {\varphi}Y,{\varphi}X)
{g}(  Z,\bar{\xi}) \\
&  \quad +2{g}(  Z,\bar{\eta}(  Y)
{\varphi}^{2}X).
\end{align*}
Then, we deduce $(  {\nabla}_{X}{\varphi})  (  Y)  +(
{\nabla}_{{\varphi}X}{\varphi})  (  {\varphi
}Y)  =2{g}(  {\varphi}X,{\varphi}Y)  \bar{\xi
}+\bar{\eta}(  Y)  {\varphi}^{2}(  X)  -{\eta
}^{\alpha}(  Y)  h_{\alpha}(  X)$.
Obviously,  $\bar{\eta}(X)=\sum_{\alpha=1}^{r}\varepsilon_{\alpha}\eta^{\alpha}(X)=\sum_{\alpha=1}^{r}g(X,\xi_{\alpha})=g(X,\bar{\xi})$.
 \end{proof}
\begin{corollary}
Let $(
{M},{\varphi},{\xi}_{\alpha},{\eta}^{\alpha},{g})$ be
an indefinite almost $\mathcal{S}$-manifold. Then, for any $X,Y\in\mathfrak{D}$:
\begin{enumerate}
\item[a)]$({\nabla}_{X}%
{\varphi})  (  Y)  +(  {\nabla}_{{\varphi}%
X}{\varphi})  (  {\varphi}Y)  =2{g}(
X,Y)  \bar{\xi}$,
\item[b)]$(
{\nabla}_{X}{\varphi})  (  {\varphi}X)  =(
{\nabla}_{{\varphi}X}{\varphi})  (  X)$.
\end{enumerate}
\end{corollary}
 \begin{proof}
The first statement follows from the above proposition.
Putting $Y:={\varphi}X$ in a), we have $(  {\nabla}_{X}{\varphi})  (  {\varphi}X)
+(  {\nabla}_{{\varphi}X}{\varphi})  (
{\varphi}^{2}X)  =2{g}(  X,{\varphi}X)  \bar{\xi
}=0$,
therefore, being ${\varphi}^{2}X=-X$, we obtain $(  {\nabla}%
_{X}{\varphi})  (  {\varphi}X)  =(  {\nabla
}_{{\varphi}X}{\varphi})  (  X).$%
 \end{proof}
\begin{remark}
\label{CondS}\emph{The statement b) can be written as ${\nabla}_{X}(  {\varphi}^{2}X)
-{\varphi}(  {\nabla}_{X}{\varphi}X)  ={\nabla
}_{{\varphi}X}(  {\varphi}X)  -{\varphi}(
{\nabla}_{{\varphi}X}X)$,
i.e. as ${\nabla}_{X}X+{\nabla}_{{\varphi}X}(  {\varphi}X)
={\varphi}[  {\varphi}X,X]$.}
\end{remark}

\subsection{Indefinite $\mathcal{S}$-manifolds}
\begin{definition}\emph{
Let $(  
{M},{\varphi},{\xi}_{\alpha},{\eta}^{\alpha},{g})$ be
an indefinite metric $g.f.f$-manifold. ${M}$ is said an {\it indefinite $\mathcal{S}$-manifold} if it is a normal indefinite almost
$\mathcal{S}$-manifold.}
\end{definition}
\begin{proposition}
\label{PropCondS}Let
$(  {M},{\varphi},{\xi}_{\alpha},{\eta}^{\alpha},
{g})  $ be an indefinite almost $\mathcal{S}$-manifold. Then ${M}$
is an indefinite $\mathcal{S}$-manifold if and only if, for any $X,Y\in\Gamma(  T{M})  $,  the Levi-Civita connection
satisfies:
\begin{align*}
(  {\nabla}_{X}{\varphi})  Y  & ={g}(  X,Y)
\bar{\xi}-\bar{\eta}(  Y)  X-  \varepsilon_{\alpha}{\eta}^{\alpha}(
X)  {\eta}^{\alpha}(  Y) \bar{\xi}%
+\bar{\eta}(  Y)  {\eta}^{\alpha}(  X)  {\xi
}_{\alpha},
\end{align*}
or equivalently%
\begin{equation}
(  {\nabla}_{X}{\varphi})  Y={g}(  {\varphi
}X,{\varphi}Y) \bar{\xi}+\bar{\eta}(  Y)
{\varphi}^{2}(  X)  .\label{eq07}%
\end{equation}
\end{proposition}
 \begin{proof}
Assuming that ${M}$ is an indefinite $\mathcal{S}$-manifold, (\ref{eq00})
becomes
$${g}(  (  {\nabla}_{X}{\varphi})  Y,Z)={g}(  {\varphi}Y,{\varphi}X)  \bar{\eta}(
Z)  -{g}(  {\varphi}Z,{\varphi}X)  \bar{\eta
}(  Y) ={g}(  Z,{g}(  {\varphi}Y,{\varphi}X)
\bar{\xi}+\bar{\eta}(  Y)  {\varphi}^{2}X),$$ 
from which
$$(  {\nabla}_{X}{\varphi})  Y ={g}(
{\varphi}X,{\varphi}Y)  \bar{\xi}+\bar{\eta}(
Y)  {\varphi}^{2}(  X)={g}(  X,Y)  \bar{\xi}-\varepsilon_{\alpha}{\eta}^{\alpha}(  X)  {\eta}^{\alpha
}(  Y)  \bar{\xi} -\bar{\eta}(  Y)  X+\bar{\eta
}(  Y) {\eta}^{\alpha}(  X)  {\xi}_{\alpha}.$$
Vice versa, we suppose that ${\nabla}$ satisfies (\ref{eq07}). Then we obtain
${g}(  (  {\nabla}_{X}{\varphi})  Y,Z)
={g}(  {\varphi}Y,{\varphi}X) \bar{\eta}(
Z)  -{g}(  {\varphi}Z,{\varphi}X)  \bar{\eta
}(  Y)$, 
and comparing with (\ref{eq00}), we deduce for any
$X,Y\in\Gamma(  T{M})$,  ${g}(  N(  Y,Z)  ,{\varphi}X)  =0$.
From Proposition \ref{Propabb}, we obtain that $N(  Y,Z)  =0$ for any
$Y,Z\in\Gamma(T {M})$, that is ${M}$ is normal.%
 \end{proof}
\begin{remark}
\label{nablaS-mani}\emph{In an
indefinite $\mathcal{S}$-manifold $(  {M},{\varphi},{\xi
}_{\alpha},{\eta}^{\alpha},{g})  $, the operators $\mathcal{L}_{{\xi}_{\alpha}}%
{\varphi}$, and then $h_{\alpha}$, vanish. In fact, by direct computation for any $X\in\Gamma(  T
  {M})  $ and for any $\alpha\in\{  1,\ldots,r\}$
we get $N(  {\varphi}X,{\xi}_{\alpha})=(  \mathcal{L}_{{\xi}_{\alpha}}%
{\varphi})  X=2h_{\alpha}(  X),$
and the normality condition implies $h_{\alpha}=0$.
Using Proposition \ref{nablax}, we obtain, for any $\alpha\in\{
1,\ldots,r\}  $,
${\nabla}_{X}{\xi}_{\alpha}=-\varepsilon_{\alpha}{\varphi}X$.}
\end{remark}

Now, we give the condition of indefinite $\mathcal{S}%
$-manifold in terms of the fundamental $2$-form:
\begin{proposition}
\label{fi S-cond}Let
$(  {M},{\varphi},{\xi}_{\alpha},{\eta}^{\alpha},
{g})  $ be an indefinite almost $\mathcal{S}$-manifold. Then ${M}$
is an indefinite $\mathcal{S}$-manifold if and only if for any $X,Y,Z\in
\Gamma(  T{M})  $:
\begin{align}
(  {\nabla}_{X}\Phi)  (  Y,Z)  
&=\bar{\eta}(Y){g}({\varphi}X,{\varphi}Z)-\bar{\eta}(Z){g}({\varphi}X,{\varphi}Y) . \label{condfi}
\end{align}
\end{proposition}
 \begin{proof} One simply uses $(  {\nabla}_{X}\Phi) (  Y,Z) ={g}(  Y,(  {\nabla}_{X}{\varphi})  Z)$
in (\ref{eq07}).
 \end{proof}
\begin{proposition}
Let $(  {M},{\varphi},{\xi}_{\alpha},{\eta}^{\alpha},
{g})  $ be an indefinite metric $g.f.f$-manifold. If the vector fields ${\xi}_{\alpha}$ are Killing, $\mathcal{L}_{\xi_{\alpha}}\eta^{\beta}=0$ for any $\alpha,\beta\in\{1,\ldots,r\}$ and ${M}$ satisfies $(\ref{eq07})$ or equivalently $(\ref{condfi})$, then ${M}$ is an indefinite $\mathcal{S}$-manifold.
\end{proposition}
\begin{proof}
Being $3d\Phi(X,Y,Z)=\mathfrak{S}_{X,Y,Z}({\nabla}_{X}\Phi)(Y,Z)$, from (\ref{condfi}) we get $
d\Phi=0$ and $(\mathcal{L}_{{\xi}_{\alpha}}\Phi)(X,Y)=0 $, since $\mathcal{L}_{{\xi}_{\alpha}}\Phi=i_{{\xi}_{\alpha}}d\Phi+di_{{\xi}_{\alpha}}\Phi$.
Proposition \ref{Nphi3} implies 
$(\mathcal{L}_{{\xi}_{\alpha}}{g})(X,\varphi Y)+{g}(X,(\mathcal{L}_{{\xi}_{\alpha}}{\varphi})Y)=0$,
 for any $\alpha\in\{1,\ldots,r\}$ and $X,Y\in\Gamma(T{M})$.
Hence, being ${\xi}_{\alpha}$ a Killing vector field, we find
$\mathcal{L}_{{\xi}_{\alpha}}{\varphi}=0$ and then $\eta^{\beta}([\xi_{\alpha},\varphi Y])=0$, for any $\alpha,\beta\in\{1,\ldots,r\}$.
In these hypotheses, (\ref{eq-01}) becomes 
\begin{align*}
2{g}(({\nabla}_{X}{\varphi})Y,Z)   &
={g}(N(Y,Z),{\varphi}X)
+2\varepsilon_{\alpha}[d{\eta}^{\alpha}({\varphi}Y,Z){\eta}^{\alpha}(X)-d{\eta
}^{\alpha}({\varphi}Z,Y){\eta}^{\alpha}(X)\\
&\quad +d{\eta
}^{\alpha}({\varphi}Y,X)  {\eta}^{\alpha}(Z)-d{\eta
}^{\alpha}({\varphi}Z,X)  {\eta}^{\alpha}(
Y)].
\end{align*}
On the other hand, (\ref{condfi}) implies
${g}(Y,({\nabla}_{X}{\varphi})Z) =\bar{\eta}(Y){g}({\varphi}X,{\varphi}Z)-\bar{\eta}(Z){g}({\varphi}X,{\varphi}Y)$,
therefore we deduce
\begin{align*}
{g}(N(Y,Z),{\varphi}X)&=
-2\varepsilon_{\alpha}[(d{\eta}^{\alpha}({\varphi}Y,Z) -d{\eta
}^{\alpha}({\varphi}Z,Y)){\eta}^{\alpha}(X)+(d{\eta
}^{\alpha}({\varphi}Y,X)-{g}({\varphi}X,{\varphi}Y))  {\eta}^{\alpha}(Z)\\
& \quad -(d{\eta
}^{\alpha}({\varphi}Z,X)-{g}({\varphi}X,{\varphi}Z))  {\eta}^{\alpha}(
Y)].
\end{align*}
Putting $Y={\xi}_{\beta}$ in the above equation, we get
\begin{equation}
{g}(N({\xi}_{\beta},Z),{\varphi}X)=
2\varepsilon_{\beta}(d{\eta
}^{\beta}({\varphi}Z,X)-{g}({\varphi}X,{\varphi}Z)).\label{Nphi1}
\end{equation}
Since $N({\xi}_{\beta},Z)=-[{\xi}_{\beta},Z]-{\varphi}[{\xi}_{\beta},{\varphi}Z]+{\xi}_{\beta}({\eta}^{\alpha}(Z)){\xi}_{\alpha}$, then
${\varphi}N({\xi}_{\beta},Z)=(\mathcal{L}_{{\xi}_{\alpha}}{\varphi})Z-{\eta}^{\alpha}[{\xi}_{\beta},{\varphi}Z]{\xi}_{\alpha}=0$ and
(\ref{Nphi1}) gives
$d{\eta
}^{\beta}({\varphi}Z,X)={g}({\varphi}X,{\varphi}Z)=\Phi({\varphi}Z,X)$.
Finally, $\mathcal{L}_{\bar{\xi}_{\alpha}}\bar{\eta}^{\beta}=0$ implying $i_{\bar{\xi}_{\alpha}} d\bar{\eta}^{\beta}=0$ and being $Y=-{\varphi}^{2}Y+{\eta}^{\alpha}(Y){\xi}_{\alpha}$, for any $Y\in\Gamma(T{M})$, we obtain $d{\eta}^{\beta}(Y,X)=-d{\eta}^{\beta}({\varphi}^{2}Y,X)+{\eta}^{\alpha}(Y)d{\eta}^{\beta}
	({\xi}_{\alpha},X)=-\Phi({\varphi}^{2}Y,X)=\Phi(Y,X)$.
 Then ${M}$ is an indefinite almost $\mathcal{S}$-manifold and we apply Proposition \ref{PropCondS}. 
\end{proof}

\section{Examples of indefinite $\mathcal{S}$-manifolds}

We describe some examples of indefinite $\mathcal{S}$-manifolds, where the characteristic vector fields are either timelike or spacelike or of both types.

\begin{example}\emph{
We consider $\mathbb{R}^{6}$ with its standard coordinates $\{x^{1},x^{2},y^{1},y^{2},z^{1},z^{2}\}$. We introduce on $\mathbb{R}^{6}$ an indefinite $g.f.f$-structure $(\varphi,\xi_{1},\xi_{2},\eta^{1},\eta^{2},g)$ by setting
\begin{align*}
\xi_{\alpha}=\frac{\partial}{\partial z^{\alpha}}, \qquad \eta^{\alpha}=dz^{\alpha}-\sum_{i=1}^{2}y^{i}dx^{i},\qquad \alpha\in\{1,2\},
\end{align*}
\begin{align*}
	g=-\sum_{\alpha=1}^{2}\eta^{\alpha}\otimes\eta^{\alpha}+\frac{1}{2}\sum_{i=1}^{2}((dx^{i})^{2}+(dy^{i})^{2}),
\end{align*}
and $\varphi$ given, with respect to the frame $\{\frac{\partial}{\partial x^{1}},\frac{\partial}{\partial x^{2}},\frac{\partial}{\partial y^{1}},\frac{\partial}{\partial y^{2}},\xi_{1},\xi_{2}\}$, by the matrix
\begin{equation*}
F=
\left(
\begin{array}{ccc}
0    & I_{2}     & 0 \\
-I_{2} & 0  & 0\\
0  & Y  & 0
\end{array}
\right),\qquad {\rm where}\quad Y=
\left(
\begin{array}{cc}
y^{1} & y^{2} \\
 y^{1} & y^{2}
\end{array}
\right).
\end{equation*}
We put $M=(\mathbb{R}^{6}_{2},\varphi,\xi_{1},\xi_{2},\eta^{1},\eta^{2},g)$.
A straightforward computation shows that $g$ is a metric tensor field. Firstly we check that $g$ is non-degenerate and then we compute its index. The matrix $G$ of $g$ is given by
\begin{equation*}
G=
\left(
\begin{array}{cccccc}
\frac{1}{2}-2(y^1)^2 & -2y^1y^2           & 0           & 0           & y^1 & y^1 \\
-2y^1y^2           & \frac{1}{2}-2(y^2)^2 & 0           & 0           & y^2 & y^2 \\
0                    & 0                    & \frac{1}{2} & 0           & 0   & 0   \\
0                    & 0                    & 0           & \frac{1}{2} & 0   & 0   \\
y^1                  & y^2                  & 0           & 0           & -1  & 0 \\
y^1                  & y^2                  & 0           & 0           & 0   & -1 
\end{array}
\right),
\end{equation*}
and $detG=\frac{1}{16}\neq 0$.
Now, to determine the index of $g$, we look for the eigenvalues of $G$. Since
\begin{align*}
det(G-\lambda I)=-(\frac{1}{2}-\lambda)^{3}(1+\lambda)(\lambda^{2}+(2(y^{1})^{2}+2(y^{2})^{2}+\frac{1}{2}) \lambda -\frac{1}{2}),
\end{align*}
we find that the index of $g$ is two; therefore $g$ is a semi-Riemannian metric of the index $2$ on $\mathbb{R}^{6}$. We remark that $\xi_{1}$ and $\xi_{2}$ are timelike vector fields.
It is easy to prove that $M$ is an indefinite $\mathcal{S}$-manifold.}
\end{example}
\begin{example}\emph{
The second example of an indefinite $\mathcal{S}$-manifold is $M=(\mathbb{R}^{6}_{2},\varphi,\xi_{\alpha},\eta^{\alpha},g)$, where, for any $\alpha \in \{1,2\}$, we put
\begin{equation*}
	\xi_{\alpha}:=\frac{\partial}{\partial z^{\alpha}}, \qquad \eta^{\alpha}:=dz^{\alpha}-\sum_{i=1}^{2} \tau_{i}y^{i}dx^{i},
\end{equation*}
$\varphi$, $g$ are given by
\begin{equation*}
F=
\left(
\begin{array}{ccc}
0    & I_{2}     & 0 \\
-I_{2} & 0  & 0\\
0  & Y  & 0
\end{array}
\right),\quad {\rm where} \quad Y=
\left(
\begin{array}{cc}
-y^{1} & y^{2} \\
 -y^{1} & y^{2}
\end{array}
\right),
\end{equation*}
and
\[
g=\sum\nolimits_{\alpha=1}^{2}\eta^{\alpha}\otimes\eta^{\alpha}+\frac{1}{2}\sum\nolimits_{i=1}^{2}\tau_{i}((dx^{i})^{2}+(dy^{i})^{2}),
\]
respectively, where $\tau_{i}=\mp1$ according to whether $i=1$ or $i=2$. 
Moreover, the symmetric $(0,2)$-type tensor field $g$ is a semi-Riemannian metric because
$detG=\frac{1}{16}\neq 0$.
Therefore $g$ is non degenerate, and 
\[
det(G-\lambda I)=-(\frac{1}{2}+\lambda)^{2}(\frac{1}{2}-\lambda)(\lambda-1)(\lambda^{2}-(\frac{3}{2}+2(y^{1})^{2}+2(y^{2})^{2})\lambda+\frac{1}{2}),
\]
so, since the signs of eigenvalues are independent from the coordinates, the index of $g$ is constant. We note that in this example $\xi_{1}$ and $\xi_{2}$ are spacelike.
One proves that $M$ is an indefinite $\mathcal{S}$-manifold.}
\end{example}
\begin{example}\label{terzo}\emph{
The third example is $M=(\mathbb{R}^{4}_{1},\varphi,\xi_{1},\xi_{2},\eta^{1},\eta^{2},g)$ constructed as follows. Denoting the standard coordinates with $\{x,y,z^{1},z^{2}\}$, we endow $\mathbb{R}^{4}$ with the structure $(\varphi,\xi_{1},\xi_{2},\eta^{1},\eta^{2},g)$ where
\[
\xi_{\alpha}=\frac{\partial}{\partial z^{\alpha}}, \quad \eta^{\alpha}=dz^{\alpha}+ydx,
\]
for any $\alpha\in\{1,2\}$ and where the tensor fields $\varphi$ and $g$ are given by
\begin{equation*}
F:=\left(
\begin{array}{cccc}
0&-1&0&0\\
1&0&0&0\\ 
0&y&0&0\\
0&y&0&0
\end{array}
\right)\qquad
G:=\left(
\begin{array}{cccc}
\frac{1}{2}&0&y&-y\\
0&\frac{1}{2}&0&0\\ 
y&0&1&0\\
-y&0&0&-1
\end{array}
\right)
\end{equation*}
respectively. 
An immediate computation shows that $g$ is non-degenerate and its index is constant. In fact, we have
$detG=-\frac{1}{4}$,
and
\begin{equation*}
det(G-\lambda I)=(\frac{1}{2}-\lambda)(\lambda^{3}-\frac{1}{2}\lambda^{2}-(2y^{2}+1)\lambda+\frac{1}{2}),
\end{equation*}
hence $detG\neq0$ and, using Cartesio's rule, we deduce that the index is $1$. Therefore, the tensor field $g$ is a Lorentzian metric. Now, we observe that $\xi_{1}$ is a spacelike vector field while $\xi_{2}$ is a timelike vector field. 
One can check that $M$ is an indefinite $\mathcal{S}$-manifold.}
\end{example}
\section{Sectional curvature and ${\varphi}$-sectional curvature}
In this section,  we look for some results about the sectional curvature of indefinite $\mathcal{S}$-manifolds. Following the notations in (\cite{KN}), for the curvature tensor ${R}$ we have
${R}(  X,Y,Z)  ={\nabla}_{X}{\nabla}_{Y}Z-{\nabla
}_{Y}{\nabla}_{X}Z-{\nabla}_{[  X,Y]  }Z,$ and
${R}(  X,Y,Z,W)  ={g}(  {R}(  Z,W,Y)
,X)$,
for any $X,Y,Z,W\in\Gamma(  T{M})  $. 

A two-dimensional subspace $\pi$ of the tangent space $T_{p}{M}$ is called
\emph{non-degenerate} if and only if we have
$\Delta(  \pi)  ={g}_{p}(  X,X)  {g}_{p}(
Y,Y)$-${g}_{p}(  X,Y)  ^{2}\neq0$ for any basis $\{X,Y\}$ of $\pi$. We know that if $\pi$
is a non-degenerate $2$-plane of $T_{p}{M}$ then we can define the
\emph{sectional curvature }$K_{p}(  \pi)  $ at $p$ with respect
to the $2$-plane $\pi$, putting
\[
K_{p}(  \pi)  =\frac{{R}_{p}(  X,Y,X,Y)  }%
{\Delta(  \pi)  }=\frac{{g}_{p}(  {R}_{p}(
X,Y,Y)  ,X)  }{\Delta(  \pi)  } , \]
where $\pi=span\{  X,Y\}  $. In the following we denote
$K_{p}(  \pi)  =K_{p}(  X,Y)  $.
\begin{proposition}\label{ker}
In an indefinite $\mathcal{S}$-manifold $(  {M},{\varphi
},{\xi}_{\alpha},{\eta}^{\alpha},{g})$ one has:
\begin{itemize}
\item [a)] the distribution $\ker{\varphi}$ is integrable and flat;
\item [b)] the sectional curvatures $K(  X,{\xi}_{\alpha})=\varepsilon_{\alpha}  $, for any $\alpha\in\{  1,\ldots,r\}  $, and non lightlike $X\in
\operatorname{Im}{\varphi}$.
\end{itemize}
\end{proposition}
\begin{proof}
For $X,Y\in\ker{\varphi}$ we have $X=f^{\alpha}{\xi}_{\alpha}$, $Y=t^{\beta}{\xi}_{\beta}$ then $[X,Y]=[f^{\alpha}{\xi}_{\alpha},t^{\beta}{\xi}_{\beta}]=f^{\alpha}{\xi}_{\alpha}(
	t^{\beta}){\xi}_{\beta}-t^{\beta}{\xi}_{\beta}(f^{\alpha}){\xi}_{\alpha}\in\ker{\varphi}$
and $\ker{\varphi}$ is integrable. Furthermore, since ${\nabla}_{{\xi}_{\alpha}}{\xi}_{\beta}=0$ and $[{\xi}_{\alpha},{\xi}_{\beta}]=0$, we have ${R}({\xi}_{\alpha},{\xi}_{\beta},{\xi}_{\gamma})=0$ and $\ker{\varphi}$ is flat. Note that a) holds also for indefinite almost $\mathcal{S}$-manifolds. Now, being ${M}$ an indefinite $\mathcal{S}$-manifold, we know that $
{\nabla}_{X}{\xi}_{\alpha}=-\varepsilon_{\alpha}{\varphi}X$,
$\mathcal{L}_{{\xi}_{\alpha}}{\varphi}=0$ and we have
\begin{align*}
{R}(  {\xi}_{\alpha},X,{\xi}_{\beta})   
& =-\varepsilon_{\beta}{\nabla}_{{\xi}_{\alpha}}(  {\varphi
}X)  +\varepsilon_{\beta}{\varphi}[  {\xi}_{\alpha
},X]  =\varepsilon_{\beta}(  {\varphi}[  {\xi}_{\alpha
},X]  -[  {\xi}_{\alpha},{\varphi}X]  -{\nabla
}_{{\varphi}X}{\xi}_{\alpha}) =\varepsilon_{\beta}\varepsilon_{\alpha}{\varphi}^{2}X.
\end{align*}
So, for $X\in\operatorname{Im}{\varphi}$, $X$ non lightlike, we have
$K(  X,{\xi}_{\alpha})=-\frac{\varepsilon_{\alpha}{g}(  {\varphi}^{2}X,X)  }%
{{g}(  X,X)  }=\varepsilon_{\alpha}$. 

\end{proof}

As usual, we say that a $2$-plane $\pi$ in $T_{p}{M}$, $p\in{M}$, is a ${\varphi}$-\emph{plane} if 
$\pi=span\{  X,{\varphi}X\}$ with $X\in\mathfrak{D}_{p}$, and the sectional curvature at $p$ of such a plane, with  $X$ a non lightlike vector, is said the ${\varphi}$-\emph{sectional curvature} at $p$ and is denoted by 
$H_{p}(  X)  $.

We shall prove that on an indefinite $\mathcal{S}$-manifold, as in the Sasakian
case, the ${\varphi}$-sectional curvatures determine the sectional
curvatures. 

As in \cite{Bl2}, we define a tensor field of type (0,4) given for any
$X,Y,Z,W$ in $\Gamma(  T{M})  $ by
\begin{align*}
P(  X,Y;Z,W)   & =\Phi(  X,Z)  {g}(
Y,W)  -\Phi(  X,W)  {g}(  Y,Z)  -\Phi(  Y,Z)  {g}(  X,W)  +\Phi(  Y,W)
{g}(  X,Z).
\end{align*}
The following lemmas can be easily proved.
\begin{lemma}
\label{lemma not}Let $(  {M},{\varphi},{\xi
}_{\alpha},{\eta}^{\alpha},{g})  $ be an indefinite
$\mathcal{S}$-manifold. Then:
\begin{itemize}
\item[a)] $P(
X,Y;Z,W)  =-P(  Z,W;X,Y)  $, for any $X,Y,Z,W\in\Gamma(  T{M})  $,
\item[b)] $P(  X,Y;X,{\varphi}Y)  ={g}(  X,{\varphi
}Y)  ^{2}+{g}(  X,Y)  ^{2}-\varepsilon_{X}\varepsilon
_{Y}$, where $X,Y$ are unit vector fields in $\mathfrak{D}$ and $\varepsilon_{X}={g}(  X,X)  $ and $\varepsilon
_{Y}={g}(  Y,Y)  $.
\end{itemize}
\end{lemma}
\begin{proposition}
\label{compatibil con fi}Let $(  {M},{\varphi},{\xi
}_{\alpha},{\eta}^{\alpha},{g})  $ be an indefinite
$\mathcal{S}$-manifold. Then, putting $\varepsilon=\sum_{\alpha=1}^{r}\varepsilon_{\alpha}$, for any $X,Y,Z,W\in\Gamma(  T{M})$
\begin{align*}
&{g}(  {R}(  X,Y,{\varphi}Z)  ,W)
+{g}(  {R}(  X,Y,Z)  ,{\varphi}W)  =-\varepsilon P(  X,Y;Z,W)  -Q(  X,Y;Z,W) 
\end{align*}
where  
\begin{align*}
Q(  X,Y;Z,W)   & ={g}(  W,{\varphi}Y)  (
\varepsilon(  {g}(X,Z)-{g}({\varphi}X,{\varphi}Z))-
\bar{\eta}(  Z)\bar{\eta}(  X)  ) \\
& \quad -{g}(  W,{\varphi}X)  (  \varepsilon(  {g}(Y,Z)-{g}({\varphi}
Y,{\varphi}Z))-\bar{\eta}(  Z)\bar{\eta}(  Y)  ) \\
& \quad -{g}(  Z,{\varphi}Y)  (  \varepsilon(  {g}(X,W)-{g}({\varphi}X,{\varphi}W))  -\bar{\eta}(  X)  \bar{\eta}(
W)  ) \\
& \quad +{g}(  Z,{\varphi}X)  (  \varepsilon(  {g}(Y,W)-{g}({\varphi}Y,{\varphi}W)) -\bar{\eta}(  Y)  \bar{\eta}(
W)).
\end{align*}
 Moreover if
$X,Y,Z,W\in\mathfrak{D}$ then obviously $Q(  X,Y;Z,W)  =0$ and the
following statements hold:
\begin{itemize}
\item [a)] ${g}({R}({\varphi}X,{\varphi}Y,{\varphi
   }Z)  ,{\varphi}W)  ={g}(  {R}(
   X,Y,Z)  ,W) $;
\item [b)]
${g}(  {R}(  X,{\varphi}X,Y)  ,{\varphi
}Y)   ={g}(  {R}(  X,Y,X)  ,Y)  +
{g}(  {R}(  X,{\varphi}Y,X)  ,{\varphi}Y)-2\varepsilon P(  X,Y,X,{\varphi}Y)$;
\item [c)]
${g}(  {R}(  {\varphi}X,Y,{\varphi}X)
,Y)  ={g}(  {R}(  X,{\varphi}Y,X)
,{\varphi}Y)$.
\end{itemize}
\end{proposition}%
\begin{remark}
\emph{We remark that $\varepsilon$ can vanish only if $r$ is an even number and the number of timelike characteristic vector fields is equal to the number of spacelike characteristic vector fields. Moreover, $\varepsilon=0$ means that ${g}(\bar{\xi},\bar{\xi})=0$, i.e. $\bar{\xi}=\sum_{\alpha=1}^{r}\xi_{\alpha}$ is a lightlike vector field. }
\end{remark}

We put 
\[
B(  X,Y)  ={g}(  {R}(  X,Y,X),Y), \qquad X,Y\in\Gamma(TM)\]
and 
\[D(  X)  =B(  X,{\varphi}X), \qquad X\in\Gamma(  \mathfrak{D}).\]
The following Lemma, of which we omit the long proof, gives the useful expression of $B(
X,Y)  $, for any $X,Y\in\Gamma(
\mathfrak{D})  $.
\begin{lemma}
Let $(  {M},{\varphi},{\xi}_{\alpha},{\eta}^{\alpha}%
,{g})  $ be an indefinite $\mathcal{S}$-manifold. Then, for any
$X,Y\in\Gamma(   \mathfrak{D})  $,
\begin{align}
B(  X,Y)   & =\frac{1}{32}\{  3D(  X+{\varphi
}Y)  +3D(  X-{\varphi}Y)  -D(  X+Y)
\label{espr B}\\
& \quad   -D(  X-Y)  -4D(  X)  -4D(  Y)
+24\varepsilon P(  X,Y;X,{\varphi}Y)  \}  .\nonumber%
\end{align}
\end{lemma}
Using the previous Lemmas it is possible to compute the sectional curvature of a non degenerate 2-plane $\pi=span\{  X,Y\}$ of
$\mathfrak{D}_{p}$, as follows.  
\begin{proposition}\label{ultimo}
Let $(  {M},{\varphi},{\xi}_{\alpha},{\eta}^{\alpha}%
,{g})  $ be an indefinite $\mathcal{S}$-manifold and $p$ in
${M}$. We consider a non degenerate 2-plane $\pi=span\{  X,Y\}$ of
$\mathfrak{D}_{p}$, where $X$ and $Y$ are unit vectors of $\mathfrak{D}_{p}$.
Then the sectional curvature $K_{p}(  X,Y)$ is given by
\begin{align*}
K_{p}(  X,Y)   & =\frac{1}{32(  \varepsilon_{X}\varepsilon
_{Y}-{g}(  X,Y)  ^{2})  }\{  3(  \varepsilon
_{X}+\varepsilon_{Y}+2{g}(  X,{\varphi}Y)  )
^{2}H_{p}(  X+{\varphi}Y)  \\
& \quad +3(  \varepsilon_{X}+\varepsilon_{Y}-2{g}(  X,{\varphi
}Y)  )  ^{2}H_{p}(  X-{\varphi}Y)  -(  \varepsilon_{X}+\varepsilon_{Y}+2{g}(  X,Y)
)  ^{2}H_{p}(  X+Y) \\
& \quad -(  \varepsilon_{X}+\varepsilon_{Y}-2{g}(  X,Y)
)  ^{2}H_{p}(  X-Y)  -4H_{p}(  X)  -4H_{p}(
Y) \\
& \quad   +24\varepsilon(  {g}(  X,{\varphi}Y)
^{2}+{g}(  X,Y)  ^{2}-\varepsilon_{X}\varepsilon_{Y})
\}  .
\end{align*}
\end{proposition}%
\begin{proof}
We note that if $X\in\mathfrak{D}_{p}$ we have $D_{p}(  X)  =B_{p}(  X,{\varphi}X)  ={g}%
_{p}(  {R}_{p}(  X,{\varphi}X,X)  ,{\varphi
}X)  =-{g}_{p}(  X,X)  ^{2}H_{p}(  X)$ and if $X$ and $Y$ are unit vectors of $\mathfrak{D}_{p}$, we find $$
{g}(  X+{\varphi}Y,X+{\varphi}Y)  =\varepsilon
_{X}+\varepsilon_{Y}+2{g}(  X,{\varphi}Y), \quad
{g}(  X+Y,X+Y)  =\varepsilon_{X}+\varepsilon_{Y}+2
{g}(  X,Y).$$
Being $\Delta(  \pi)=\varepsilon_{X}\varepsilon_{Y}-{g}_{p}(
X,Y)  ^{2}$, we get $K_{p}(  \pi)  =-{g}_{p}(  {R}_{p}(
X,Y,X)  ,Y) /\Delta(  \pi)=-B_{p}(
X,Y) /\Delta(  \pi)$.
Then, using (\ref{espr B}) and Lemma \ref{lemma not}, we get the required formula.
\end{proof}
 \begin{remark}
\label{Tensor1,3} \emph{We note that if $X\in\Gamma(  \mathfrak{D})
$ is a unit vector field we have 
\[
{R}(  {\xi}_{\alpha},X,
{\xi}_{\beta})  =-\varepsilon_{\beta}\varepsilon_{\alpha}X, \qquad 
{R}(  X,{\xi}_{\alpha},X)  =-\varepsilon_{X}\varepsilon
_{\alpha}\bar{\xi}.
\]
In fact, if $Y\in\Gamma(  T{M})  $, for any $\alpha\in\{  1,\ldots,r\}  $,
we have
\begin{align*}
{g}(  {R}(  X,{\xi}_{\alpha},X)  ,Y)   &
=-{g}(  {R}(  X,Y,{\xi}_{\alpha})  ,X) =\varepsilon_{\alpha}{g}(  {\nabla}_{X}(  {\varphi
}Y)  -{\nabla}_{Y}(  {\varphi}X)  -{\varphi
}[  X,Y]  ,X) \\
& =\varepsilon_{\alpha}{g}(  (  {\nabla}_{X}{\varphi
})  Y-(  {\nabla}_{Y}{\varphi})  X,X)  =\varepsilon_{\alpha}{g}(  -\bar{\eta}(  Y)  X-\bar
{\eta}(  X) \varphi^{2} Y,X) \\
& =-\varepsilon_{X}\varepsilon_{\alpha}\bar{\eta}(  Y)
   =-\varepsilon_{X}\varepsilon_{\alpha}%
{g}(  \bar{\xi},Y)  .
\end{align*}
Finally, if $X,Y\in\Gamma(\mathfrak{D})$ and $Z\in\Gamma(T{M})$ then we get
\[	{g}({R}(X,{\xi}_{\alpha},Y),Z)=-\varepsilon_{\alpha}{g}(Y,X)\bar{\eta}(Z)=-\varepsilon_{\alpha}{g}(Y,X){g}(\bar{\xi},Z).
\]}
\end{remark}
\begin{theorem}\label{richiamo}
The ${\varphi}$-sectional curvatures completely determine the sectional curvatures of
an indefinite $\mathcal{S}$-manifold.
\end{theorem}%
\begin{proof}
We show that for any $p\in{M}$ and for any non degenerate 2-plane $\pi=span\{X,Y\}$ in
$T_{p}(  {M})  $ the sectional curvature $K_{p}(
X,Y)  $ is uniquely determined by the
${\varphi}$-sectional curvature. In the sequel of the proof we suppose
that $p\in{M}$ is fixed. If $X,Y\in\mathfrak{D}_{p}$, then we apply the previous Proposition and if $X$ or $Y$ is ${\xi}_{\alpha}$, for
any $\alpha\in\{1,\ldots,r\}$, we have already seen that $K_{p}(  X,Y)
=\varepsilon_{\alpha}$. If $X,Y\in T_{p}{M}$, they
can be written in the following way:
\[
X=aZ+{\eta}^{\alpha}(X){\xi}_{\alpha},\quad Y=bW+{\eta}^{\alpha}(Y){\xi}_{\alpha},
\]
where $Z,W\in\mathfrak{D}$, ${g}_{p}(  Z,Z)
=\varepsilon_{Z} $, ${g}_{p}(  W,W)  =\varepsilon_{W}$, and $a$ and $b$ must satisfy:
\[
a^{2}\varepsilon_{Z}=\varepsilon
_{X}-\varepsilon_{\alpha}(  {\eta
}^{\alpha}(  X))^{2},\quad b^{2}\varepsilon
_{W}=\varepsilon_{Y}-\varepsilon_{\alpha}(
{\eta}^{\alpha}(  Y)  )  ^{2}.
\] 
Therefore, we compute
\begin{align}
{g}_{p}({R}_{p}&( X,Y,X),Y)=a^{2}b^{2}{g}_{p}({R}_{p}( Z,W,Z),W)+2a^{2}b\:{\eta}^{\beta}(Y){g}_{p}({R}_{p}( Z,W,Z),{\xi}_{\beta})\label{futuro2}\\
& +2ab^{2}{\eta}^{\alpha}(X){g}_{p}({R}_{p}( Z,W,{\xi}_{\alpha}),W)+2ab{\eta}^{\alpha}(X){\eta}^{\beta}(Y){g}_{p}({R}_{p}( Z,W,{\xi}_{\alpha}),{\xi}_{\beta})\nonumber\\
& +a^{2}{\eta}^{\beta}(Y){\eta}^{\delta}(Y){g}_{p}({R}_{p}( Z,{\xi}_{\beta},Z),{\xi}_{\delta})+2ab{\eta}^{\beta}(Y){\eta}^{\alpha}(X){g}_{p}({R}_{p}( Z,{\xi}_{\beta},{\xi}_{\alpha}),W)\nonumber\\
&+2a{\eta}^{\beta}(Y){\eta}^{\alpha}(X){\eta}^{\delta}(Y){g}_{p}({R}_{p}( Z,{\xi}_{\beta},{\xi}_{\alpha}),{\xi}_{\delta})+b^{2}{\eta}^{\alpha}(X){\eta}^{\gamma}(X){g}_{p}({R}_{p}( {\xi}_{\alpha},W,{\xi}_{\gamma}),W)\nonumber\\
& +2b{\eta}^{\alpha}(X){\eta}^{\beta}(Y){\eta}^{\gamma}(X){g}_{p}({R}_{p}( {\xi}_{\alpha},Z,{\xi}_{\gamma}),{\xi}_{\beta})+{\eta}^{\alpha}(X){\eta}^{\beta}(Y){\eta}^{\gamma}(X){\eta}^{\delta}(Y){g}_{p}({R}_{p}({\xi}_{\alpha},{\xi}_{\beta},{\xi}_{\gamma}),{\xi}_{\delta}).\nonumber
\end{align}
Now, separately we take the terms of previous expression into account, using Remark
\ref{Tensor1,3} and the Bianchi identity, as follows:
{\setlength\arraycolsep{2pt}
\begin{eqnarray*}
{g}_{p}({R}_{p}( Z,W,Z),{\xi}_{\beta})&=&{g}_{p}({R}_{p}( Z,{\xi}_{\beta},Z),W)=-\varepsilon_{Z}\varepsilon_{\beta}{g}_{p}(\bar{\xi},W)=0,\\
{g}_{p}({R}_{p}( Z,W,{\xi}_{\alpha}),W)&=&{g}_{p}({R}_{p}({\xi}_{\alpha},W, Z),W)={g}_{p}({R}_{p}(W,{\xi}_{\alpha},W),Z)=-\varepsilon_{W}\varepsilon_{\alpha}{g}_{p}(\bar{\xi},Z)=0,\\
{g}_{p}({R}_{p}( Z,W,{\xi}_{\alpha}),{\xi}_{\beta})&=&-{g}_{p}({R}_{p}( Z,{\xi}_{\alpha},{\xi}_{\beta}),W)-{g}_{p}({R}_{p}( Z,{\xi}_{\beta}),{\xi}_{\alpha}),W)={g}_{p}({R}_{p}({\xi}_{\alpha},Z,{\xi}_{\beta}),W)\\
&\quad& +\varepsilon_{\beta}{g}_{p}(Z,W){g}_{p}( \bar{\xi}),{\xi}_{\alpha})=-\varepsilon_{\beta}\varepsilon_{\alpha}{g}_{p}(Z,W)+\varepsilon_{\beta}\varepsilon_{\alpha}{g}_{p}(Z,W)=0,\\
{g}_{p}({R}_{p}( Z,{\xi}_{\beta},{\xi}_{\alpha}),W)&=&-{g}_{p}({R}_{p}( Z,{\xi}_{\beta},W){\xi}_{\alpha})=\varepsilon_{\beta}{g}_{p}(Z,W){g}_{p}( \bar{\xi}),{\xi}_{\alpha})=\varepsilon_{\beta}\varepsilon_{\alpha}{g}_{p}(Z,W),\\
{g}_{p}({R}_{p}( Z,{\xi}_{\beta},{\xi}_{\alpha}),{\xi}_{\delta})&=&-{g}_{p}({R}_{p}( {\xi}_{\beta},Z,{\xi}_{\alpha}),{\xi}_{\delta})=\varepsilon_{\beta}\varepsilon_{\alpha}{g}_{p}(Z,{\xi}_{\delta})=0,\\
{g}_{p}({R}_{p}( {\xi}_{\alpha},W,{\xi}_{\gamma}),{\xi}_{\beta})&=&\varepsilon_{\gamma}\varepsilon_{\alpha}{g}_{p}(Z,{\xi}_{\beta})=0.
\end{eqnarray*}}
Therefore, replacing the previous expressions in (\ref{futuro2}), we
have:
\begin{align*}
{g}_{p}({R}_{p}( X,Y,X),Y)
&=a^{2}b^{2}{g}_{p}({R}_{p}( Z,W,Z),W)-a^{2}\varepsilon_{Z}\bar{\eta}(Y)\bar{\eta}(Y)\\
&\quad +2ab\bar{\eta}(Y)\bar{\eta}(X){g}_{p}( Z,W) -b^{2}\varepsilon_{W}\bar{\eta}(X)\bar{\eta}(X).\nonumber
\end{align*}
Hence, being $K_{p}(X,Y)=-\varepsilon_{X}\varepsilon_{Y}{g}_{p}({R}_{p}( X,Y,X),Y)$, we deduce
\begin{align}\label{numero}
K_{p}(  X,Y) &=\varepsilon_{X}\varepsilon_{Y}\{  a^{2}%
b^{2}{g}_{p}(  {R}_{p}(  Z,W,W)  ,Z) -2ab\bar{\eta}(  Y)  \bar{\eta}(  X)  {g}_{p}(
Z,W) \\
&\quad +b^{2}\varepsilon_{W}\bar{\eta}(  X)
^{2}+a^{2}\varepsilon_{Z} \bar{\eta}(  Y)
^{2}\}\nonumber.
\end{align}
Now, we note that
$${g}_{p}(  Z,W) =\frac{1}{ab}{g}_{p}(  X-{\eta
}^{\alpha}(  X)  {\xi}_{\alpha},Y-{\eta}^{\beta}(
Y)  {\xi}_{\beta}) +{\eta}^{\alpha}(
X)  {\eta}^{\beta}(  Y)  {g}_{p}(  {\xi
}_{\alpha},{\xi}_{\beta})  \}  =-\frac{1}{ab}\varepsilon_{\alpha}{\eta
}^{\alpha}(  X)  {\eta}^{\alpha}(  Y)  ,$$
\begin{align*}
	{g}_{p}({R}_{p}(Z,W,W),Z)&=[\epsilon_{Z}\epsilon_{W}-{g}_{p}(Z,W)^{2}]K_{p}(Z,W)\\
	&=\frac{1}{a^{2}b^{2}}[a^{2}\epsilon_{Z}b^{2}\epsilon_{W}-\left(\varepsilon_{\alpha}{\eta}^{\alpha}(X)
	{\eta}^{\alpha}(Y)\right)^{2}]K_{p}(Z,W)\\
	&=\frac{1}{a^{2}b^{2}}[\left(\epsilon_{X}-\varepsilon_{\alpha}{\eta}^{\alpha}(X)^{2}\right)\left(\epsilon_{Y}-\varepsilon_{\alpha}{\eta}^{\alpha}(Y)^{2}\right)\\
	& \quad -\left(\varepsilon_{\alpha}{\eta}^{\alpha}(X)
	{\eta}^{\alpha}(Y)\right)^{2}]K_{p}(Z,W).
\end{align*}
Thus, (\ref{numero}) becomes
\begin{align*}
K_{p}(  X,Y)   & =\varepsilon_{X}\varepsilon_{Y}\{  [
(  \varepsilon_{X}-\varepsilon_{\alpha
}(  {\eta}^{\alpha}(  X)  )  ^{2})(  \varepsilon_{Y}-\varepsilon_{\beta
}(  {\eta}^{\beta}(  Y)  )  ^{2}) \\
& \quad   -(\varepsilon_{\alpha}{\eta
}^{\alpha}(  X)  {\eta}^{\alpha}(  Y)  )
^{2}]  K_{p}(  Z,W) +2\bar{\eta}(  Y)  \bar{\eta}(  X)\varepsilon_{\alpha}{\eta}^{\alpha}(
X)  {\eta}^{\alpha}(  Y)\\
& \quad +(  \varepsilon_{Y}-\varepsilon_{\beta
}(  {\eta}^{\beta}(  Y)  )  ^{2})  
\bar{\eta}(  X) ^{2}+(  \varepsilon_{X}-\varepsilon
_{\alpha}(  {\eta}^{\alpha}(  X)  )  ^{2})
 \bar{\eta}(  Y) ^{2}\},
\end{align*}
and this completes the proof, since $K_{p}(Z,W)$ is given as in Proposition \ref{ultimo}.%
\end{proof}

We recall the following result.
\begin{lemma}[\cite{[O'N]}]
\label{lemma R=S} Let $(  V,g)
$ be a semi-Euclidean vector space and $R$ a $(0,4)$-type tensor on
$V$ such that for any $X,Y,Z,W\in V$ the following conditions hold:
\begin{itemize}
\item[a)] $R(  X,Y,Z,W)  =-R(  Y,X,Z,W)  ,$
\item[b)] $R(  X,Y,Z,W)  =-R(  X,Y,W,Z)  ,$
\item[c)] $R(  X,Y,Z,W)  =R(  Z,W,X,Y)  ,$
\item[d)] $\mathfrak{S}_{Y,Z,W}R(  X,Y,Z,W)  =0.$
\end{itemize}
\noindent If $R(  X,Y,X,Y)  =0$ for any linearly independent and non lightlike vectors $X,Y\in V$, then $R=0$.
Moreover, if $R$ and $S$ are $(0,4)$-type tensors on $V$ such
that the conditions (a-d) are satisfied and $R(
X,Y,X,Y)  =S(  X,Y,X,Y)  $ for any $X,Y\in V$ linearly
independent non lightlike vectors, then $R=S$.
\end{lemma}
 \begin{proposition}
\label{Lemma 04}Let $(  {M},{\varphi},{\xi}_{\alpha
},{\eta}^{\alpha},{g})  $ be an indefinite $\mathcal{S}%
$-manifold, $T$ and $S$ be $(0,4)$-type tensor fields on ${M}$ such that the
following conditions hold:
\begin{itemize}
\item [i)] $ T(X,Y,Z,W)=-T(Y,X,Z,W), \;\; S(
X,Y,Z,W)  =-S(  Y,X,Z,W)$, \;\;  $X,Y,Z,W\in\Gamma(T{M})$ 
\item[ii)] $T(  X,Y,Z,W)  =-T(  X,Y,W,Z), \;\; S(
X,Y,Z,W)  =-S(  X,Y,W,Z)$, \;\; $X,Y,Z,W\in\Gamma(  T{M})  $
\item [iii)] $T(  X,Y,Z,W)  =T(  Z,W,X,Y), \;\; S(
X,Y,Z,W)  =S(  Z,W,X,Y)$, \;\; $X,Y,Z,W\in\Gamma(  T{M})  $
\item[iv)] $\mathfrak{S}_{Y,Z,W}T(  X,Y,Z,W)  =0, \quad  \mathfrak{S}%
_{Y,Z,W}S(  X,Y,Z,W)  =0$, \quad $X,Y,Z,W\in\Gamma(  T{M})  $
\item[v)] for any $X,Y,Z,W\in\Gamma(  \mathfrak{D})  $%
\begin{align*}
T(  X,Y,{\varphi}Z,W)  +T(  X,Y,Z,{\varphi}W)
& =\varepsilon P(  X,Y;Z,W) \\
S(  X,Y,{\varphi}Z,W)  +S(  X,Y,Z,{\varphi}W)
& =\varepsilon P(  X,Y;Z,W)
\end{align*}
\item [vi)]for any $X,Y\in\Gamma(  \mathfrak{D})  $ and for any
$\alpha,\beta,\gamma,\delta\in\{  1,\ldots,r\}  $
\begin{itemize}
\item [(a)] $T(  X,{\xi}_{\alpha},X,Y)  =S(  X,{\xi}_{\alpha},X,Y) $,
\item[(b)] $T(  {\xi}_{\alpha},X,{\xi}_{\beta},Y)  =S(
{\xi}_{\alpha},X,{\xi}_{\beta},Y)  $,
\item[(c)] $T(  {\xi}_{\alpha},X,{\xi}_{\beta},{\xi}_{\gamma
})  =S(  {\xi}_{\alpha},X,{\xi}_{\beta},{\xi}_{\gamma
}) $ ,
\item[(d)] $T(  {\xi}_{\alpha},{\xi}_{\beta},{\xi}_{\gamma
},{\xi}_{\delta})  =S(  {\xi}_{\alpha},{\xi}_{\beta},{\xi}_{\gamma
},{\xi}_{\delta}) $ .
\end{itemize}
\end{itemize}

Then, if $T(  X,{\varphi}X,X,{\varphi}X)  =S(
X,{\varphi}X,X,{\varphi}X)  $ for any $X\in\Gamma(
\mathfrak{D})  $ non lightlike vector field, one has $T=S$.
\end{proposition}
\begin{proof}
It is to verify that \emph{v)}
implies that for any $X^{\prime},Y^{\prime},Z^{\prime},W^{\prime}$ in
$\Gamma(  \mathfrak{D})  $ 
\[
T(  {\varphi}X^{\prime},{\varphi}Y^{\prime},{\varphi
}Z^{\prime},{\varphi}W^{\prime})  =T(  X^{\prime},Y^{\prime
},Z^{\prime},W^{\prime}),
\]
and, using the above formula, we obtain
$T(  {\varphi}X^{\prime},{\varphi}Y^{\prime},Z^{\prime},W^{\prime
})  =T(  X^{\prime},Y^{\prime},{\varphi}Z^{\prime},{\varphi}%
W^{\prime})$.
Analogously, for the tensor field $S$ we have
$S(  {\varphi}X^{\prime},{\varphi}Y^{\prime},Z^{\prime},W^{\prime
})  =S(  X^{\prime},Y^{\prime},{\varphi}Z^{\prime},{\varphi}%
W^{\prime})$.

Now, being ${\varphi}_{p}$ an almost complex structure on $\mathfrak{D}%
_{p}$ for any $p\in{M}$,  from a well-known result analogous to
the Lemma \ref{lemma R=S} (\cite{[BR]}), in the case of a real vector
space endowed with an almost complex structure, we deduce
$T(  X^{\prime},Y^{\prime},Z^{\prime},W^{\prime})  =S(
X^{\prime},Y^{\prime},Z^{\prime},W^{\prime})$.
Then, in particular, we have
$T(  X^{\prime},Y^{\prime},X^{\prime},Y^{\prime})  =S(
X^{\prime},Y^{\prime},X^{\prime},Y^{\prime})$.

Now, if $X,Y\in\Gamma(  T{M})$ are linearly independent and non lightlike, 
we compute $T(  X,Y,X,Y)  $ and $S(  X,Y,X,Y)  $,
writing $X=X^{\prime}+{\eta}^{\alpha}(  X)  {\xi}_{\alpha}$
and $Y=Y^{\prime}+{\eta}^{\alpha}(  Y)  {\xi}_{\alpha}$, and likewise to (\ref{futuro2}), by the $\mathfrak{F}(M)$-linearity of $T$ and $S$, using \emph{vi)}, we get $T(  X,Y,X,Y)  =S(  X,Y,X,Y)  $.%
 \end{proof}
\begin{remark}
\emph{Using Remark \ref{Tensor1,3} and Proposition \ref{ker}, the Riemannian (0,4)-type curvature tensor field $
{R}$ satisfies the properties listen in Proposition \ref{Lemma 04}. Thus, 
 it is uniquely determined by the $\varphi$-sectional curvature.}
\end{remark}
\begin{theorem}
Let $(  {M},{\varphi},{\xi}_{\alpha},{\eta
}^{\alpha},{g})  $ be an indefinite $\mathcal{S}$-manifold. Then the
${\varphi}$-sectional curvature $c$ is pointwise constant, $c\in
\mathfrak{F}(  {M})  $, if and only if the Riemannian
$(0,4)$-type curvature tensor field ${R}$ is given by
\begin{align}
	{R}(  X,Y,Z,W)   &  =-\frac{c+3\varepsilon}{4}\{  
{g}({\varphi}Y,{\varphi}Z)  {g}({\varphi} X,{\varphi}W) -{g}(  {\varphi}X,{\varphi}Z)
{g}( {\varphi}Y,{\varphi}W)\} \label{equivalente}\\
&\quad-\frac{c-\varepsilon}{4}\{  \Phi(  W,X)
\Phi(  Z,Y)     -\Phi(  Z,X)  \Phi(  W,Y) +2\Phi(  X,Y)  \Phi(  W,Z)
\} \nonumber\\
&  \quad -\{ \bar {\eta}(  W)  \bar {\eta
}(  X)  {g}(  {\varphi}Z,{\varphi
}Y)  -\bar {\eta}(  W)  \bar {\eta}(
Y)  {g}(  {\varphi}Z,{\varphi}X)  +\bar {\eta
}(  Y)  \bar {\eta}(  Z)  {g}(
{\varphi}W,{\varphi}X)\nonumber\\
&  \quad    -\bar {\eta}(  Z)  \bar {\eta}(  X)  {g}(  {\varphi}W,{\varphi}Y)
\} .\nonumber
\end{align}
\end{theorem}
\begin{proof}
We suppose that the ${\varphi}$-sectional curvature $c$ is pointwise
constant and in order to prove (\ref{equivalente}), denote by $S(  X,Y,Z,W) $ the right-hand side of (\ref{equivalente}). 
Obviously $S$ is a tensor field of type (0,4) on ${M}$, and we shall
prove that $S$ coincides with ${R}$. To this end it is easy to check that for
any $X,Y,Z,W\in\Gamma(  T{M})  $ we have the properties of
skew-symmetry $-S(  X,Y,W,Z)  =S(  X,Y,Z,W)
=-S(  Y,X,Z,W)$
and the Bianchi identity $
\mathfrak{S}%
_{Y,Z,W}S(  X,Y,Z,W)  =0,$
while the property \emph{iii)} of Proposition
\ref{Lemma 04}, $S(  X,Y,Z,W)  =S(  Z,W,X,Y)  $, follows
by the Bianchi identity and the skew-symmetries.

Now, for $X,Y,Z,W\in\Gamma(  \mathfrak{D})  $, computing $S(  X,Y,Z,{\varphi}W) +S(  X,Y,{\varphi}Z,W)$ we get
\begin{align*}
&S(X,Y,Z,{\varphi}W) +S(  X,Y,{\varphi}Z,W)=
-\frac{c}{4}\{ {g}(Y,Z)\Phi(X,W)-{g}(X,Z)\Phi(Y,W)+\Phi(Y,Z){g}(X,W)\\
&\quad -\Phi(X,Z){g}(Y,W) +{g}(W,X)\Phi(  Z,Y)-\Phi(  Z,X)  {g}(W,Y)+\Phi(W,X){g}(Z,Y)-{g}(Z,X)\Phi(  W,Y)\} \\
&\quad -\frac{\varepsilon}{4}\{ 3\Phi( X,W){g}( Z,Y)-3\Phi(  Y,W)  {g}(X,Z) +3{g}(  X,W)\Phi(Y,Z)-3{g}(Y,W)  \Phi(X,Z)\\
& \quad +\Phi( Y,Z){g}(  W,X)-\Phi(  X,Z)  {g}(  W,Y)+\Phi( X,W){g}(  Z,Y)-\Phi(  Y,W)  {g}(Z,X)  \}\\
&=-\varepsilon\{ \Phi( X,W){g}( Z,Y) -\Phi(X,Z){g}(Y,
W)-\Phi(  Y,W)  {g}(X,Z)+{g}(  X,W)  \Phi(Y,Z) \}\\
& =\varepsilon P(X,Y;Z,W).
\end{align*}
We continue verifying \emph{vi)} of Proposition \ref{Lemma 04}, and obtaining $S(  X,{\xi}_{\alpha},X,Y)=0={R}(  X,{\xi}_{\alpha},X,Y)$, $S(  {\xi}_{\alpha},X,{\xi}_{\beta},{\xi}_{\gamma})=0={R}(  {\xi}_{\delta},X,{\xi}_{\beta},{\xi}_{\gamma
})$, $S(  {\xi}_{\alpha},{\xi}_{\delta},{\xi}_{\beta},{\xi}_{\gamma})=0={R}(  {\xi}_{\delta},{\xi}_{\delta},{\xi}_{\beta},{\xi}_{\gamma
})$ and
\begin{align*}
S(  {\xi}_{\alpha},X,{\xi}_{\beta},Y)   &  =-\frac{c+3\varepsilon}{4}\{  
{g}({\varphi}X,{\varphi}{\xi}_{\beta})  {g}({\varphi}  {\xi}_{\alpha},{\varphi}Y)-{g}(  {\varphi} {\xi}_{\alpha},{\varphi}{\xi}_{\beta})
{g}( {\varphi}X,{\varphi}Y)\}\\
&\quad  -\frac{c-\varepsilon}{4}\{  \Phi(  Y, {\xi}_{\alpha})
\Phi(  {\xi}_{\beta},X)-\Phi(  {\xi}_{\beta}, {\xi}_{\alpha})  \Phi(  Y,X) +2\Phi(   {\xi}_{\alpha},X)  \Phi(  Y,{\xi}_{\beta})
\} \\
&  \quad-\{  \bar{\eta}(  Y)  \bar{\eta
}(   {\xi}_{\alpha})  {g}(  {\varphi}{\xi}_{\beta},{\varphi
}X) -{\bar{\eta}}(  Y)  \bar{\eta}(
X)  {g}(  {\varphi}{\xi}_{\beta},{\varphi} {\xi}_{\alpha})+{\bar{\eta}
}(  X)  \bar{\eta}(  {\xi}_{\beta})  {g}(
{\varphi}Y,{\varphi} {\xi}_{\alpha}) \\
&  \quad -\bar{\eta}(  {\xi}_{\beta})  \bar{\eta}(   {\xi}_{\alpha})  {g}(  {\varphi}Y,{\varphi}X)
\}=\varepsilon_{\alpha}\varepsilon_{\beta}{g}(  X,Y)  =
{R}(  {\xi}_{\alpha},X,{\xi}_{\beta},Y) .
\end{align*}
For any $X\in\Gamma(  \mathfrak{D})$ non lightlike vector field, we compute $S(
X,{\varphi}X,X,{\varphi}X)  $, obtaining:
\begin{align}
S(  X,{\varphi}X,X,{\varphi}X)   &  =-\frac{c+3\varepsilon}{4}\{  
{g}({\varphi}^{2}X,{\varphi}X)  {g}({\varphi} X,{\varphi}^{2}X) -{g}(  {\varphi}X,{\varphi}X)
{g}( {\varphi}^{2}X,{\varphi}^{2}X)\}\label{corr01}\\
&\quad  -\frac{c-\varepsilon}{4}\{  \Phi(  {\varphi}X,X)
\Phi(  X,{\varphi}X)   -\Phi(  X,X)  \Phi({\varphi}X,{\varphi}X) +2\Phi(  X,{\varphi}X)  \Phi(  {\varphi}X,X)
\} \nonumber\\
&  \quad -\{  \bar{\eta}(  {\varphi}X)  \bar{\eta
}(  X)  {g}(  {\varphi}X,{\varphi
}^{2}X)  -\bar{\eta}(  {\varphi}X)  \bar{\eta}(
{\varphi}X)  {g}(  {\varphi}X,{\varphi}X)  \nonumber\\
&  \quad  +\bar{\eta
}(  {\varphi}X)  \bar{\eta}(  X)  {g}(
{\varphi}^{2}X,{\varphi}X)  -\bar{\eta}(  X)  \bar{\eta}(  X)  {g}(  {\varphi}^{2}X,{\varphi}^{2}X)
\}\nonumber\\
&=\frac{c+3\varepsilon}{4}{g}(  X,X)^{2}-\frac{c-\varepsilon}{4}\{-{g}(  X,X)^{2}-2{g}(  X,X)  ^{2}\}\nonumber\\
&=\frac{c+3\varepsilon}{4}{g}(  X,X)^{2}+3\frac{c-\varepsilon}{4}{g}(  X,X)^{2}=c{g}(  X,X)  ^{2}.\nonumber
\end{align}
Moreover, since by definition of ${\varphi}$-sectional
curvature we have
 \begin{equation}
{R}(  X,{\varphi}X,X,{\varphi}X)  =c{g}(
X,X)  ^{2}.\label{C02}%
\end{equation}
from (\ref{corr01}) and (\ref{C02}) we get
${R}(  X,{\varphi}X,X,{\varphi}X)  =S(
X,{\varphi}X,X,{\varphi}X)$,
and, using Proposition \ref{Lemma 04}, the previous Remark and the properties
of the tensor field $S$, we obtain ${R}(  X,Y,Z,W)  =S(  X,Y,Z,W)$,
for any $X,Y,Z,W\in\Gamma(  T{M})  $, that is the formula
(\ref{equivalente}).

Conversely, if we assume (\ref{equivalente}), choosing a
point $p\in{M}$ and a ${\varphi}$-plane $\pi=span\{
X,{\varphi}X\}  $, with $X\in\mathfrak{D}_{p}$ non lightlike vector, by direct computation, omitting the point $p$, we have
\[
H(  X) =\frac{c+3\varepsilon}{4{g}(  X,X)  ^{2}}{g}(  X,X)^{2}+3\frac{c-\varepsilon}{4{g}(  X,X)  ^{2}}{g}(  X,X)^{2}=c.
\]
\end{proof}
\section{Sectional Curvature in the case $\varepsilon=0$, an example}
In this section we consider the case $\varepsilon=0$, as already pointed out, $r=2p$ and ${\xi}_{1},\ldots,{\xi}_{p}$ are timelike vector field, ${\xi}_{p+1},\ldots,{\xi}_{2p}$ are spacelike vector field. We call such a manifold a \emph{special indefinite} $\mathcal{S}$-\emph{manifold}. Let $(  {M},{\varphi},{\xi
}_{\alpha},{\eta}^{\alpha},{g})  $ be a special indefinite
$\mathcal{S}$-manifold. The tensor $Q$ is given by
\begin{align*}
Q(  X,Y;Z,W)   & =-{g}(  W,{\varphi}Y) 
\bar{\eta}(  Z)
\bar{\eta}(  X)+{g}(  W,{\varphi}X)\bar{\eta}(  Z) \bar {\eta}(
Y) +{g}(  Z,{\varphi}Y)\bar{\eta}(  X)  \bar{\eta}(
W)\\
& \quad -{g}(  Z,{\varphi}X)\bar{\eta}(  Y)  \bar{\eta}(
W),
\end{align*}
and 
\begin{align*}
{g}(  {R}(  X,Y,{\varphi}Z)  ,W)
+{g}(  {R}(  X,Y,Z)  ,{\varphi}W)=-Q(  X,Y;Z,W)
\end{align*}
Moreover, being $Q(  X,Y;Z,W)  =0$ for any
$X,Y,Z,W\in\mathfrak{D}$,  we have
\begin{itemize}
\item [a)] ${g}({R}({\varphi}X,{\varphi}Y,{\varphi
   }Z)  ,{\varphi}W)  ={g}(  {R}(
   X,Y,Z)  ,W) $ ;
\item [b)]
${g}(  {R}(  X,{\varphi}X,Y)  ,{\varphi
}Y)   ={g}(  {R}(  X,Y,X)  ,Y)  +
{g}(  {R}(  X,{\varphi}Y,X)  ,{\varphi}Y)$  ;
\item [c)]
${g}(  {R}(  {\varphi}X,Y,{\varphi}X)
,Y)  ={g}(  {R}(  X,{\varphi}Y,X)
,{\varphi}Y)$  .
\end{itemize}
Furthermore, for $X,Y\in\Gamma(   \mathfrak{D})  $
\begin{align*}
B(  X,Y)  =\frac{1}{32}\{  3D(  X+{\varphi
}Y)  +3D(  X-{\varphi}Y)  -D(  X+Y)-D(  X-Y)  
-4D(  X)  -4D(  Y)  \},
\end{align*}
and for a non degenerate 2-plane $\pi=span\{  X,Y\}$ of
$\mathfrak{D}_{p}$, where $X$ and $Y$ are unit vectors of $\mathfrak{D}_{p}$,
\begin{align*}
K_{p}(  X,Y)   & =\frac{1}{32(  \varepsilon_{X}\varepsilon
_{Y}-{g}(  X,Y)  ^{2})  }\{  3(  \varepsilon
_{X}+\varepsilon_{Y}+2{g}(  X,{\varphi}Y)  )
^{2}H_{p}(  X+{\varphi}Y)  \\
& \quad +3(  \varepsilon_{X}+\varepsilon_{Y}-2{g}(  X,{\varphi
}Y)  )  ^{2}H_{p}(  X-{\varphi}Y) -(  \varepsilon_{X}+\varepsilon_{Y}+2{g}(  X,Y)
)  ^{2}H_{p}(  X+Y) \\
& \quad -(  \varepsilon_{X}+\varepsilon_{Y}-2{g}(  X,Y)
)  ^{2}H_{p}(  X-Y)  -4H_{p}(  X)  -4H_{p}(
Y)\} .
\end{align*}
Finally we have that the
${\varphi}$-sectional curvature $c$ is pointwise constant, $c\in
\mathfrak{F}(  {M})  $, if and only if the Riemannian
(0,4)-type curvature tensor field ${R}$ is given by
\begin{align}
	{R}(  X,Y,Z,W)   &  =-\frac{c}{4}\{  
{g}({\varphi}Y,{\varphi}Z)  {g}({\varphi} X,{\varphi}W)-{g}(  {\varphi}X,{\varphi}Z)
{g}( {\varphi}Y,{\varphi}W) \label{equivalente2}\\
&\quad  + \Phi(  W,X)
\Phi(  Z,Y)   -\Phi(  Z,X)  \Phi(  W,Y) +2\Phi(  X,Y)  \Phi(  W,Z)
\} \nonumber\\
&  \quad -\{  \bar{\eta}(  W)  \bar{\eta
}(  X)  {g}(  {\varphi}Z,{\varphi
}Y)  -\bar{\eta}(  W)  \bar{\eta}(
Y)  {g}(  {\varphi}Z,{\varphi}X)  \nonumber\\
&  \quad  +\bar{\eta
}(  Y)  \bar{\eta}(  Z)  {g}(
{\varphi}W,{\varphi}X)  -\bar{\eta}(  Z)  \bar{\eta}(  X)  {g}(  {\varphi}W,{\varphi}Y)
\} \nonumber.
\end{align}

 An example of a special indefinite $\mathcal{S}$-manifold is $M=(\mathbb{R}^{4}_{1},\varphi,\xi_{1},\xi_{2},\eta^{1},\eta^{2},g)$, which is described in Example \ref{terzo}. 
We observe that the metric is Lorentzian, $\xi_{1}$ is a spacelike vector field while $\xi_{2}$ is a timelike vector field, then, since $\varepsilon=0$, the structure is a special indefinite $\mathcal{S}$-structure.
Now, we compute the tensor field $Q$ on some relevant set of vector fields, the sectional curvature and $\varphi$-sectional curvature. We know that $Q=0$ on $\mathfrak{D}$, moreover we have
{\setlength\arraycolsep{2pt}
\begin{eqnarray}
	Q(\xi_{1},Y;Z,W)&=&	-Q(\xi_{2},Y;Z,W)= -{g}(W,\varphi Y)\bar{\eta}(Z)+{g}(Z,\varphi Y)\bar{\eta}(W)=0,\nonumber\\
   Q(\xi_{\alpha},Y;\xi_{\beta},W)&=&Q(Y,\xi_{\alpha};W,\xi_{\beta})=-\varepsilon_{\alpha}\varepsilon_{\beta}g(W,\varphi Y),\label{Q01}
\end{eqnarray}}for any $Y,Z,W\in\Gamma(\mathfrak{D})$ and for any $\alpha,\beta\in\{1,2\}$.
Equation (\ref{Q01}) shows that $Q$ never vanishes.
Now,  computing the Christoffel's symbols we obtain:
$$\Gamma_{12}^{3}=\Gamma_{12}^{4}=\frac{1}{2},\quad \Gamma_{13}^{2}=-\Gamma_{14}^{2}=-\Gamma_{23}^{1}=\Gamma_{24}^{1}=-1,\quad
\Gamma_{23}^{3}=\Gamma_{23}^{4}=-\Gamma_{24}^{3}=-\Gamma_{24}^{4}=-y,$$
whereas the other $\Gamma_{ij}^{k}$ vanish. To compute the $\varphi$-sectional curvature, being $\mathfrak{D}$ globally spanned by $X=\frac{\partial}{\partial x}-y\xi_{1}-y\xi_{2}$ and $Y=\varphi X=\frac{\partial}{\partial y}$, we value $H(X)$. So, we have
\begin{align*}
	R(X,\varphi X,X)
	&=\nabla_{X}\left(\Gamma_{21}^{h}-y(\Gamma_{23}^{h}+\Gamma_{24}^{h})\frac{\partial}{\partial x^{h}}-\xi_{1}-\xi_{2}\right)-\nabla_{\xi_{1}}X-\nabla_{\xi_{2}}X\\
	&=-\frac{1}{2}\nabla_{X}(\xi_{1}+\xi_{2})-(\Gamma_{31}^{h}-y(\Gamma_{33}^{h}+\Gamma_{34}^{h})
	+\Gamma_{41}^{h}-y(\Gamma_{43}^{h}+\Gamma_{44}^{h}))\frac{\partial}{\partial x^{h}}\\
	&=[\Gamma_{11}^{h}-y(\Gamma_{31}^{h}+\Gamma_{41}^{h})-y(\Gamma_{13}^{h} -y(\Gamma_{33}^{h}+\Gamma_{43}^{h})
	+\Gamma_{14}^{h}-y(\Gamma_{34}^{h}+\Gamma_{44}^{h}))]\frac{\partial}{\partial x^{h}}
	=0,
\end{align*}
\begin{align*}
	g(X,X)&=g(\frac{\partial}{\partial x},\frac{\partial}{\partial x})-2y(g(\frac{\partial}{\partial x},\xi_{1})+g(\frac{\partial}{\partial x},\xi_{2}))+y^{2}(g(\xi_{1},\xi_{1})+g(\xi_{1},\xi_{2})+g(\xi_{2},\xi_{2}))=\frac{1}{2}.
\end{align*}
It follows that
\[
H(X)=-\frac{1}{g(X,X)^{2}}g(R(X,\varphi X,X),\varphi X)=0.
\]
Then, $M$ is an indefinite $\mathcal{S}$-space form with $c=0=\varepsilon$ and, from (\ref{equivalente2}), the Riemannian curvature tensor field ${R}$ is given by:
\begin{align}
	{R}(  X,Y,Z,W)   &  = -\{  \bar{\eta}(  W)  \bar{\eta
}(  X)  {g}(  {\varphi}Z,{\varphi
}Y)  -\bar{\eta}(  W)  \bar{\eta}(
Y)  {g}(  {\varphi}Z,{\varphi}X)  \nonumber\\
&  \quad  +\bar{\eta
}(  Y)  \bar{\eta}(  Z)  {g}(
{\varphi}W,{\varphi}X)  -\bar{\eta}(  Z)  \bar{\eta}(  X)  {g}(  {\varphi}W,{\varphi}Y)
\} .\nonumber
\end{align}

Authors address\\
 Department of Mathematics,  University of Bari \\
  Via E. Orabona 4,\\
  I-70125 Bari (Italy) \\
  {\texttt{brunetti@dm.uniba.it}\,,  \texttt{pastore@dm.uniba.it}}
\end{document}